\documentclass[11pt,oneside,reqno]{amsart}
\usepackage{amsmath,amsthm,amssymb}
\usepackage[colorlinks=true,linkcolor=blue, urlcolor=red, citecolor=blue, backref]{hyperref}
\usepackage{graphicx,color}
\usepackage{subfig}
\usepackage{xcolor}
\usepackage{hyperref}
\usepackage{mathrsfs}
\usepackage[utf8]{inputenc}
\usepackage{wasysym}
\usepackage{dsfont}
\usepackage{psfrag}
\usepackage[all]{xy}
\usepackage{tikz-cd}
\usepackage{adjustbox}
\usepackage{xcolor}
\usepackage{hyperref}
\usepackage{mathrsfs}
\usepackage[utf8]{inputenc}
\usepackage{wasysym}
\usepackage{dsfont}
\usepackage{psfrag}
\usepackage[all]{xy}
\usepackage{tikz-cd}
\usepackage{adjustbox}
\usepackage{soul}

\definecolor{Red}{cmyk}{0,1,1,0}

\definecolor{Blue}{cmyk}{1,1,0,0}

\theoremstyle{plain}
\newtheorem{theorem}{Theorem}[section]
\newtheorem{corollary}[theorem]{Corollary}
\newtheorem{proposition}[theorem]{Proposition}
\newtheorem{lemma}[theorem]{Lemma}
\newtheorem*{theorem-main}{Main Theorem}

\theoremstyle{definition}
\newtheorem{definition}[theorem]{Definition}
\newtheorem{remark}[theorem]{Remark}

\theoremstyle{definition}
\newtheorem{maintheorem}{Theorem}
\newtheorem{maincorollary}[maintheorem]{Corollary}

\newcommand{\al} {\alpha}

       \newcommand{\De}{\Delta}
\newcommand{\vep}{\varepsilon}

\newcommand{\om} {\omega}

\newcommand{\supp}{\operatorname{supp}}
\newcommand{\diam}{\operatorname{diam}}

\newcommand{\leb}{Leb}

\newcommand{\cC}{\mathcal{C}}

\newcommand{\cP}{\mathcal{P}}

\newcommand{\cE}{\mathcal{E}}

\newcommand{\cQ}{\mathcal{Q}}

\usepackage[normalem]{ulem}

\newcounter{main}

\makeatletter
\let\c@equation\c@theorem
\makeatother
\numberwithin{equation}{section}

\title[Exponential decay of correlations for Gibbs measures of Axiom A attractors]{Exponential decay of correlations for Gibbs measures \\ on attractors of Axiom A flows}

\author[D. Daltro]{Diego Daltro}
\address{Instituto de Matem\'atica e Estat\'istica, Universidade Federal da Bahia \\ 
Av. Ademar de Barros s/n, 40170-110 Salvador, Brazil}
\email{daltrodiego@gmail.com}

\author[P. Varandas]{Paulo Varandas}
\address{ CMUP, Universidade do Porto \& Departamento de Matem\'atica, Universidade Federal da Bahia\\
Av. Ademar de Barros s/n, 40170-110 Salvador, Brazil}
\email{paulo.varandas@ufba.br}


\date{\today}

\begin{document}

\begin{abstract}
In this paper we study the decay of correlations for Gibbs measures associated to codimension one Axiom A attractors for flows. 
We prove that a codimension one Axiom A attractors whose strong stable foliation is $C^{1+\alpha}$ either have exponential 
decay of correlations with respect to all Gibbs measures associated to H\"older continuous potentials or their stable and unstable bundles 
are jointly integrable. As a consequence, there exist $C^1$-open sets of $C^3$-vector fields generating Axiom A flows having attractors
so that: (i) mix exponentially with respect to equilibrium states associated with H\"older continuous potentials, (ii) their time-1 maps satisfy an almost
sure invariance principle, and (iii) the growth of the number of closed orbits of length $T$ is described by the topological entropy of the attractor.
\end{abstract}

\maketitle


\section{Introduction}

The classical thermodynamic formalism for hyperbolic diffeomorphisms and flows ensures that equilibrium states exist
and are unique for every H\"older continuous potential, restricted to every basic piece of the non-wandering set (see e.g. ~\cite{Bo75,BR75,Ru76b,Si72}).
A common feature of both discrete-time and continuous-time contexts is that equilibrium states are invariant Gibbs measures, a fact which allows one to study several of its statistical properties. However, while equilibrium states of a topologically mixing hyperbolic set for a diffeomorphism have exponential decay of correlations for H\"older continuous observables (cf.~\cite{Bo75})  there exist Axiom A flows displaying hyperbolic basic sets (suspended horseshoes)
with arbitrarily slow mixing rates ~\cite{Pol85,ruelle1983}. The following conjecture motivated many contributions in this research area:

\medskip
\noindent {\bf Conjecture~1 (Bowen-Ruelle):} \emph{Every topologically mixing Anosov flow mixes exponentially with respect to all its equilibrium states.}
\medskip

Bowen-Ruelle conjecture has been supported by the general belief in two other conjectures, 
namely that Anosov flows whose stable and unstable bundles are jointly integrable are not topologically mixing, and that Anosov flows whose  stable and unstable bundles are not jointly integrable mix exponentially fast. 
It is worth mentioning that, despite some groundbreaking advances on the exponential mixing rate of the SRB measure for contact Anosov flows, Anosov flows with smooth invariant foliations and codimension one Anosov flows and a recent proof of the Bowen-Ruelle conjecture for $C^\infty$-Anosov flows in dimension three (cf.~\cite{BW,chernov98,D,li,St0,T,TZ}), the general picture concerning the decay of correlations for Anosov flows is still far from complete.

\smallskip
Inspired by Chernov ~\cite{chernov98}, Dolgopyat introduced in the late nineties a systematic approach to the decay of correlations for flows using 
the non-integrability of invariant subbundles (measured in terms of a temporal distortion function or the Poincar\'e first return map to suitable
cross-sections). He proved that
$C^{2+\varepsilon}$ transitive Anosov flows
  whose $C^{1}$~stable and unstable foliations which are
  jointly nonintegrable mix exponentially with respect to
  the SRB measure and
H\"older continuous observables ~\cite{D},
and that superpolynomial decay of correlations is typical, in a measure theoretic sense of
prevalence, for Axiom~A flows with respect to any
equilibrium state associated to a H\"older continuous potential~\cite{dolgopyat98}. 
Building on these ideas, Field, Melbourne and
T\"or\"ok~\cite{FMT} proved that there exist $C^{2}$-open,
$C^{r}$-dense sets of $C^{r}$-Axiom~A flows ($r\geqslant 2$)
whose non-trivial basic sets mix superpolynomially. More recently,
Baladi and Vall\'ee \cite{BV}, \'Avila, Gou\"ezel and Yoccoz~\cite{AGY} and Ara\'ujo and Melbourne~\cite{AM}
established criterions to prove exponential decay of correlations for the SRB measures associated to suspension (semi)flows over piecewise 
expanding and piecewise hyperbolic maps. 
The previous criteria have some common requirements, among them the smoothness of the strong stable foliation, a uniform nonintegrability condition (UNI) on 
the roof function and the fact that the probability measure preserved by a quotiented Poincar\'e first return map satisfies the Federer property.
Recall that a measure $\nu$ on a metric space $J$ 
 satisfies the Federer property if there are $A,K>0$ such that
$\nu(B(x,Ar))\leqslant K\, \nu(B(x,r))$
for any small $r>0$ and every $x\in \supp\nu$, 
and that the Lebesgue measure satisfies this property. 
Therefore, these criteria lead to very successful applications in the context of Axiom A attractors and geometric Lorenz attractors, 
where the SRB measure is absolutely continuous with respect to the Lebesgue measure along the unstable foliation,
and the Teichm{\"u}ller flow, where the Masur-Veech invariant measure is absolutely continuous with respect to Lebesgue 
(see \cite{ABV,AM,ArVar,AGY} and references therein).
Actually, in view of the previous results yielding exponential mixing for the SRB measure of Axiom A attractors it is natural to pose the following:

\medskip
\noindent {\bf Conjecture~2:} \emph{Every topologically mixing attractor of an Axiom A flow mixes exponentially with respect to all its equilibrium states.}
\medskip

This conjecture encompasses Bowen-Ruelle conjecture, as every transitive Anosov flow is clearly an attractor, but considers also the case 
of Axiom A attractors which are compact proper subsets of the manifold. Actually, it is an abstraction of Ruelle's conjecture in \cite{P,Ru86} and, in this context, the space of equilibrium states coincides with the space invariant Gibbs measures. We observe that there exist horseshoes which may display arbitrary slow 
mixing rates (cf.~\cite{ruelle1983}). However, the geometry of attractors is substantially different, as these contain all unstable manifolds. 

\smallskip
This paper aims to contribute to the study of the correlations decay rate for Gibbs measures of Axiom A attractors. 
Until quite recently, very few answers were known concerning decay of correlations for other classes of equilibrium states rather than SRB measures. 
A first motive concerned the fact that the Federer property is not true for Gibbs measures associated to piecewise H\"older continuous potentials 
(which arise naturally from the reduction from suspension flows to the description of the base dynamics, cf. Subsection~\ref{subsec:termo}), as pointed out by Baladi and Vall\'ee in \cite[Remark 2.1]{BV}, where they point a missing argument in \cite{D,P,PS}. A second one is of technical nature, as most available methods required the smoothness (or Lipschitz continuity) of both 
strong stable and strong unstable foliations.
In \cite{Sto}, Stoyanov proved exponential decay of correlations for Gibbs measures on general hyperbolic basic sets 
under the assumptions of a local nonintegrability condition, that the strong stable foliation is uniformly Lipschitz continuous and a regular distortion property along
unstable manifolds (we refer the reader to \cite{Sto} for the definitions), using a modified version of Dolgopyat's method which does not requires the Federer property. The regular distortion property along unstable manifolds turns out to hold for Axiom A basic sets whose unstable bundle satisfies a pinching condition 
(cf. \cite{St1}), in particular it holds for codimension one basic sets. However, it remained unclear whether such equilibrium states satisfy the Federer property, or some reasonable weaker concept that still allows one to recover Dolgopyat's original strategy (cf.~\cite[page~1095]{Sto}). 
More recently, in \cite{DV} we consi\-dered suspension (semi)flows over ${C}^{1+\alpha}$ {full branch} Markov piecewise expanding interval maps 
and piecewise hyperbolic maps and roof functions satisfying the UNI condition, and proved exponential decay of correlations with respect to Gibbs measures associated to piecewise H\"older continuous potentials. The argument explored that such Gibbs measures do have uniform bounds for nested cylinder 
sets and allowed to conclude that some special classes of codimension one 
attractors for $C^{1+\alpha}$ Axiom A flows 
mix exponentially fast for every equilibrium state associated to a H\"older continuous potential. 
Some weaker Federer properties have been considered by Gou\"ezel~\cite{Go}, 
to implement variations around Dolgopyat's argument 
for hyperbolic skew-products, and by Gou\"ezel and Stoyanov~\cite{GS,St0} which considered Gibbs measures admiting Pesin sets with exponentially small tails 
and used them to study decay of correlations for contact Anosov flows (we refer the reader to these references for the definitions). 
Our approach here builds over our previous work \cite{DV}, and it is inspired by the work of Tsujii ~\cite{T0} 
on suspension semi-flows of angle-multiplying maps displaying a transversality condition 
and Butterley and War~\cite{BW} on the decay of correlations for codimension 
one Anosov flows. Ultimately, we relate the exponential decay of correlations for Gibbs measures of codimension one Axiom A attractors 
with the absence of transversality for the Poincar\'e first return time (expressed in terms of twisted transfer operators) 
and the nonintegrability of the stable and unstable subbundles.

This paper is organized as follows. In Section~\ref{sec:statements} we describe the setting and state the main results. The main result stating that codimension one Axiom A attractors with smooth stable foliation mix exponentially fast with respect to all Gibbs measures provided that the stable and unstable bundle are not
jointly integrable (Theorem~\ref{mainthm:hyp}) appears as a byproduct of a somewhat similar statement, involving the concept of transversality, 
for suspension semiflows (Theorem~\ref{thm:A}). 
In Section~\ref{geom} we describe geometric aspects of hyperbolic attractors. More precisely, we recall the symbolic dynamics of hyperbolic attractors and 
semiconjugacy to suspension semiflows, describe the relation between joint integrability of hyperbolic bundles 
and topological mixing, and construction of adapted partitions for piecewise expanding interval maps.
Twisted transfer operators appear in Section~\ref{transfer}, where we prove a Lasota-Yorke inequality
recall some preliminary results 
on Gibbs measures for piecewise expanding interval maps and transfer operators. 
Section~\ref{sec:UNI-transversal} contains some of the core estimates in the paper and may be of independent interest. Here we show that in the presence of transversality one can recover  a UNI condition, adapted to the context of suspension flows over Markov maps which may not be full branch. Indeed, 
first we 
prove that either the roof function is cohomologous to a piecewise constant one or it satisfies notion of transversality, measured in terms of the twisted transfer operators for arbitrary potentials (cf. Proposition~\ref{transv2}). 
Second, while the usual UNI condition stated
for full branch Markov maps requires the roof function to have a uniform displacement for all points (see Definition~\ref{def:UNI}), the transversality 
condition guarantees a pointwise estimate which one needs to spread for uniform size domains. In particular, while the statement of the cancellation lemma
has the same statement as the one associated to the UNI condition, the construction and constants are chosen in a somewhat unusual order (see e.g. the statement of Proposition~\ref{prop:transv-implies-wUNI}, Remark~\ref{rmk:choice-constants} and equation~\ref{eq:cts}).
Most of the results of this section hold for suspension semiflows over piecewise expanding maps in arbitrary dimension, but there are two instances where 
having one-dimensional unstable manifolds is used crucially. First, the cancellations lemma (Lemma~\ref{lemm:cancel-N}) is proved for maps in the interval and 
uses the mean value theorem. Second, as the boundary of Markov partitions may have a fractal structure \cite{Bo78} (except for two dimensional diffeomorphisms or piecewise expanding Markov interval maps, where it is formed by smooth one dimensional submanifolds and points, respectively), 
its boundary may contain cusps. In the case of the interval we use the following important fact: there exists $a>0$ so that the intersection of any ball $B$ with small radius $r$ centered at a point $x$ with the Markov subinterval $P$ containing $x$  contains a ball of radius $a \cdot r$ (actually here it is enough to take $a=1/2$). 
This fact is crucial to expand pointwise cancellations to a neighborhood of it (see item (2) in Definition~\ref{def:UNI-weak}).

In Section~\ref{sec:flows-decay-transversal} we prove the exponential decay of correlations for suspension flows over piecewise expanding Markov interval maps and roof functions which are not cohomologous to a piecewise constant one. In view of the aforementioned results, the roof function satisfies the transversality and the uniform nonintegrability conditions. 
Once more we cannot follow \cite{BW} which, as the SRB measure is absolutely continuous with respect to Lebesgue, use the transversality to prove a 
van der Corput lemma (cancellations of oscillatory integrals). This strategy does not find a counterpart in the case of other equilibrium states. For that reason, we make a bridge between the concept of transversality
and a weaker pointwise version of the uniform nonintegrability condition (adapted to piecewise expanding but not full branch piecewise expanding maps). This
allow us to obtain cancellations of oscillatory integrals by constructing a non-homogeneous partitions where all elements have roughly the same size, tailored to Gibbs measures and using a mild Federer property on sub-cylinders of the Markov partition (see Subsection~\ref{subsec:addapted-partititons}).
In particular our strategy offers an alternative proof for the cancellations associated to the SRB measure 
in \cite{BW}.

\section{Statement of the main results}\label{sec:statements}

As mentioned at the introduction our main results concern decay of correlations for hyperbolic attractors.
Let $N$ be a compact Riemannian manifold and $\|\cdot\|$ denote the Riemannian norm on the tangent space $TN$.  
Assume that $(\Phi_t)_{t\in\mathbb R}$ is a $C^1$-smooth flow. We say that a $\Phi_t$-invariant subset $\Lambda\subset N$ is a \emph{hyperbolic set} 
if there exists a $D\Phi^{t}$-invariant and continuous splitting
$T_{\Lambda}N = E^{s} \oplus E^{c} \oplus E^{u} $
where $ E^{c} $ is the one-dimensional bundle tangent to
the flow and there exists $C,\lambda >0$ such that $\|
  \left. D\Phi_t\right|_{E^{s}} \| \leqslant Ce^{-\lambda t}$,
and $\| \left. D\Phi_{-t}\right|_{E^{u}} \| \leqslant
Ce^{-\lambda t}$, for all $t\geqslant 0$.
A \emph{hyperbolic basic set} is an invariant, closed, topologically
transitive, locally maximal hyperbolic set. 
Recall that a basic set $\Lambda$ is called an
\emph{attractor} if it is transitive and there exists a neighbourhood $U$ of $\Lambda$, and $t_0>0$, such that
$\Lambda = \bigcap_{t\in  \mathbb R_+}\Phi_{t} (U)$.  
We say that a $C^1$-flow $(\Phi_t)_{t\in \mathbb R}$ is: (i) \emph{Axiom A} if its non-wandering set $\Omega$ is a hyperbolic set and it is the closure of the set of critical elements, formed by hyperbolic periodic orbits and hyperbolic singularities, and (ii) \emph{Anosov} if $\Lambda=N$ is a hyperbolic set for the flow. 
Since hyperbolic basic sets  are transitive, there exists a unique equilibrium state 
$\mu_\phi$ for the flow $(\Phi_t)_{t\in \mathbb R}$
on the attractor $\Lambda$ with respect to each  H\"older continuous potential $\phi$
(cf. \cite{BR75}). In this paper we will only consider equilibrium states $\mu_\phi$ for some
hyperbolic basic set with respect to H\"older continuous observables. 
We say that $\mu$ has \emph{exponential decay of correlations} if for every $\beta>0$ there exist constants $c,C>0$ (depending on $\mu$) 
such that 
$$
C_{\mu,v,w}(t):= \Big|\int v (w\circ\Phi_t)\,  d\mu_\phi-\int v \,d\mu_\phi \, \int w \,d\mu_\phi\Big| 
	\leqslant Ce^{-ct}\, \|v\|_{{C}^\beta}|w|_{\infty}
$$
for all $v\in{C}^\beta(N), w\in L^{\infty}(N)$ and $t>0$. The function $C_{\mu,v,w}(t)$ is called a  \textit{correlation function}.
If the stable and unstable bundles are jointly integrable then there exists a
codimension one invariant foliation which is transversal to
the flow direction, and this foliation is subfoliated
by both the stable and unstable foliations (see \cite[Proposition~1.6]{Plante}
whose proof holds for general Axiom A attractors).
Moreover, if this is the case then the flow is
(bounded-to-one) semiconjugate to a locally constant
suspension over a subshift of finite type
(see~\cite[Proposition~3.3]{FMT}), and these flows may not be topologically mixing, 
or may be topologically mixing but have mixing rates slower than exponential.

\begin{maintheorem}\label{mainthm:hyp}
\emph{Let $N$ be a compact Riemannian manifold, $(\Phi_t)_{t\in \mathbb R}$ be a $C^{1+\alpha}$ smooth flow and
 $\Lambda\subset N$ be a hyperbolic attractor such that the unstable bundle is one-dimensional and 
the stable bundle is $C^{1+\alpha}$.
If the stable and unstable foliations are not jointly integrable then 
$(\Phi_t)_{t\in \mathbb R}$ has exponential mixing with respect to every equilibrium state. }
\end{maintheorem}

Let us make some comments on the assumptions. First, joint nonintegrability is a $C^1$-open and dense condition on the flow (see e.g. \cite{FMT}).  
Second, it is worth mentioning that the regularity of the strong stable 
foliation can be obtained through a strong domination condition, which 
defines a $C^{1+\alpha}$-open condition on the space of vector fields identical to \cite[Theorem~1]{ABV}. 
In particular, combining Theorem~\ref{mainthm:hyp} with the results in \cite{ABV} we obtain the following:

\begin{maincorollary}\label{thm:main-attractors}
\emph{For each $d$-dimensional Riemannian manifold $N$ ($d\geq 3$) 
there exists a $C^{1}$-open subset of
  $C^3$-vector fields $\mathcal U \subset \mathfrak X^3(N)$ such that for
  each $X\in\mathcal U$ the associated flow is Axiom~A and exhibits
  a non-trivial codimension one attractor which mixes exponentially with
  respect to every equilibrium state}.
\end{maincorollary}

Second, in the special context of 
three-dimensional transitive Anosov flows, Plante~\cite{Plante} proved that topological mixing is equivalent to the joint nonintegrability (in general the joint nonintegrability is known to imply on topological mixing). It is natural to ask whether Plante's results extends to more general attractors: 

\medskip
\noindent {\bf Problem:} \emph{Given a three-dimensional topologically mixing Axiom A attractor whose strong stable foliation is smooth, are the strong stable and strong unstable foliations jointly non-integrable?}
\medskip

A positive answer to the previous question will validate Conjecture~2 for the class of strongly dissipative three-dimensional Axiom A attractors.

At this point it is also natural to ask whether a similar statement 
to the previous theorem can be obtained for Anosov flows. Indeed, while it is not hard to extend our results to flows whose action of the derivative on
the unstable bundle is conformal. In particular we can deduce that  $C^{2+\alpha}$-Anosov flows
 so that $\dim E^s=1$, $\dim E^u \geqslant 2$ and $D\Phi_t\mid_E^u$ is conformal and whose stable and unstable bundles are not jointly 
 integrable have exponential decay of correlations with respect to every equilibrium state. However, the set of 
 $C^1$-Anosov flows whose action on the expanding bundle is conformal has infinite codimension in the space of $C^1$-Anosov flows.  

Finally, we observe that Theorem~\ref{mainthm:hyp} and its proof provide a number of interesting consequences. We include two of them, one ergodic and one other providing geometric information. The first one concerns the 
ergodic properties of the equilibrium states. It is known that the Almost Sure 
Invariance Principle (ASIP) holds for not necessarily mixing hyperbolic basic sets of a flow \cite{DP} and to 
time-1 maps of hyperbolic and Lorenz attractors with superpolynomial mixing (cf.~\cite[Theorem~1]{MT} and \cite{AMV}). In particular, we obtain the following 
immediate consequence:

\begin{maincorollary}\label{thm:ASIP}
\emph{
Let $\mathcal U \subset \mathfrak X^3(N)$ be the $C^{1}$-open subset of $C^3$-vector fields  given by Corollary~\ref{thm:main-attractors} and 
$\Lambda_X$ denote the exponentially mixing codimension one attractor.  Let $\mu_X$ be an equilibrium state associated to a H\"older continuous potential
on $\Lambda_X$.
For each $C^\infty$ observable $\psi: N \to \mathbb R$ so that $\int \psi\, d\mu_X=0$ the ASIP holds for the time-$1$ map $\Phi_1^X$ of the flow $(\Phi^X_t)_t$: 
passing to an enriched probability space, there exists a sequence $Y_0,Y_1, Y_2, \ldots$ of i.i.d. normal random variables with mean zero and variance
 $
 \sigma_X^2:=\lim_{n\to\infty} \frac1n \int_{\Lambda_X} \Big( \sum_{j=0}^{n-1}  \psi \circ \Phi_j^X\Big)^2 \, d\mu_X
 $  
such that, for any $\delta>0$,
\[
\sum_{j=0}^{n-1}\psi \circ \Phi_j^X=\sum_{j=0}^{n-1} X_j+O(n^{1/4+\delta}),\quad a.e.
\]}
\end{maincorollary}

The second application concerns the distribution of periodic orbits of prime period and the dynamical zeta function. 
If $\Lambda\subset N$ is a hyperbolic attractor denote by $\text{co}(\Lambda)$ the set of primitive closed orbits contained in $\Lambda$. 
Moreover, if $\gamma \in \text{co}(\Lambda)$ let $\ell(\gamma)>0$ denotes its prime period.
Pollicott and Sharp~\cite{PS}
considered the dynamical zeta function 
$$
\zeta(s)=\prod_{\gamma\in \text{co}(\Lambda)}\, (1-e^{-s\ell(\gamma)})^{-1}
$$
and used contraction of twisted transfer operators (as the one proved in Section~\ref{sec:flows-decay-transversal}) to obtain sharp formulas for the asymptotic growth of the cardinality of closed orbits. Hence, as a consequence of \cite{PS} we obtain the following:

\begin{maincorollary}\label{thm:zeta}
\emph{
Let $\mathcal U \subset \mathfrak X^3(N)$ be the $C^{1}$-open subset of $C^3$-vector fields,
$\Lambda_X$ be the codimension one Axiom A attractor given by Corollary~\ref{thm:main-attractors}, and $h_X$ denote its topological entropy. 
There exists $0<c_0<h_X$ so that the zeta function $\zeta$ has analytic continuation in $\mathfrak R(s)>c_0$, except for a simple pole at $s=h_X$. Moreover, there exists $0<c<h_X$ so that 
$$
\# \Big\{ \gamma \in \text{co}(\Lambda_X) \colon \ell(\gamma) \leqslant T \Big\} = \text{li}\,(e^{h_X \,T}) + O(e^{c\, T})
	\quad\text{as} \quad T\to+\infty,
$$
where $\text{li}\,(x)=\int_2^x \frac1{\log u}\, du\sim \frac{\log x}x$ as $x\to+\infty$.} 
\end{maincorollary}

\medskip
\subsection*{Comments and ingredients in the proofs}
The main results of the paper are actually a consequence of an analogous statement for suspension semiflows.
Indeed, as hyperbolic attractors have finite Markov partitions, these are semiconjugate to suspension semiflows over piecewise expanding maps 
(cf. Subsection~\ref{sec:hyp}). 
Recall that, given a $C^{1+\al}$-piecewise expanding Markov map $T: J \to J$, we say that a (not necessarily invariant) probability measure $\nu$ is a \textit{Gibbs measure} with respect to a piecewise H\"older continuous potential $\phi:J\to\mathbb{R}$ if there are constants $C_5>0\;\text{and}\; \, P\in\mathbb{R}$ such that
\begin{eqnarray}\label{eq:Gibbs}
C_5^{-1}\leqslant \dfrac{\nu(\cP^{(n)}(x))}{\exp(-Pn+S_n\phi(y))}\leqslant C_5
\end{eqnarray}
for all $x\in J$, $y\in \cP^{(n)}(x)$ and $n\geqslant 1$. Here, as usual, $\cP^{(n)}(x)$ stands for the element of the partition  $\cP^{(n)}=\bigvee_{j=0}^{n-1} T^{-j}(\cP)$ that contains the point $x$, and $S_n\phi=\sum^{n-1}_{j=0}\phi \circ T^j$. 
Theorem~\ref{mainthm:hyp} will follow as a standard consequence of
the following dichotomy:

\begin{maintheorem}\label{thm:A}
\emph{Suppose that $T: I\to I$ is a $C^{1+\alpha}$ piecewise expanding and Markov interval map 
and that $X_t: J^r\to J^r$ is a suspension semiflow over $T$ with piecewise $C^{1+\alpha}$ roof function $r$. Then either:
\begin{enumerate}
\item $r$ is cohomologous to a piecewise constant function, or 
\item for every $T$-invariant Gibbs measure $\mu$ there exist $C,c>0$ (depending on $\mu$) such that  for any smooth observables 
$v,w\in {C}^1(J^r,\mathbb{R})$
\begin{eqnarray*}
|C_{\mu,v,w}(t)|\le Ce^{-ct}\|v\|_{{C}^1}\|w\|_{{C}^1}, \qquad \forall t>0.
\end{eqnarray*}
\end{enumerate}
Furthermore, item (2) holds for a $C^1$-generic subset of roof functions $r: J \to \mathbb R_+$.}
\end{maintheorem}

Indeed, Theorem~\ref{mainthm:hyp} will follow as a standard consequence of Theorem~\ref{thm:A} (see e.g. \cite{AM,AGY,DV} for the complete reduction via  suspension semiflows over hyperbolic skew-products). 
Theorem~\ref{thm:A} improves \cite[Theorem~B]{ABV}, where the previous statement addresses only the case of SRB measures, and \cite[Corollary~3]{DV}, 
which had the requirement that the piecewise expanding interval map was full branch. The core mechanism used in the proof of this theorem is the 
so-called transversality, considered by Butterley and War \cite{BW} and expressed in terms of transfer operators associated to the geometric potential.  

As mentioned above, there are two instances where the SRB measure played a key role in previous proofs of exponential decay of correlations. Firstly, the SRB measures satisfy
the Federer property, which ensures that cancellations of the twisted transfer operator obtained at a certain fixed scale propagate to the entire phase space. 
Gibbs measures  for the quotient dynamics $T$ and piecewise H\"older potentials may not have the Federer property but turn out to satisfy a sufficient 
condition over non-homogeneous partitions on the interval (see Subsection~\ref{subsec:addapted-partititons} for the precise description).
Secondly, if the piecewise expanding interval map $T$ is not full branch and the roof function is not cohomologous to a constant then 
the transfer operators satisfy a transversality condition (obtained in Lemma~\ref{transv2} by routine modification from a similar notion in \cite{BW}),
but this property seems not sufficient to produce cancellations on oscillatory integrals, and ultimately the contraction of twisted transfer operators.
Indeed, this kind of van der Corput lemma (see Lemma 3.14 in \cite{BW}) uses strongly that the SRB measure is absolutely continuous with 
respect to the Lebesgue measure. 
In order to overcome this fact, we modify the argument concerning exponential decay of correlations for suspension flows in \cite{AM,BV,AGY} 
replacing the uniform non-integrability condition (expressed for full branch piecewise expanding maps) by a pointwise version of the latter obtained as a consequence of transversality (see Section~\ref{sec:UNI-transversal}).

\section{Geometric aspects of hyperbolic and suspension flows}\label{geom}

\subsection{Hyperbolic basic sets}\label{sec:hyp}
In this subsection we recall some of the geometric ingredients of hyperbolic flows, namely its invariant manifolds, and consequences.
In particular we recall the existence of Markov partitions, semiconjugacy to suspension flows over subshifts of finite type and application to the thermodynamic formalism of hyperbolic flows.

\subsubsection{Basic concepts}

Let $N$ be a compact differentiable manifold and let $(\Phi_t)_{t\in \mathbb R}$ be a $C^1$-smooth flow on $N$. A compact and 
$\Phi_t$-invariant subset $\Lambda\subset N$  is called \textit{uniformly hyperbolic} if there exists a $D\Phi_t$-invariant splitting $T_\Lambda N=E^s\oplus <X>\oplus E^u$
(here $<X>$ denotes the one-dimensional subbundle tangent to the flow direction) and there are constants $c,\lambda>0$ such that:
$\|D\Phi_t(x)v\|\leqslant ce^{-\lambda t}\|v\|$ for every $v\in E_x^s,\,t\geqslant 0$, and
$\|D\Phi_{-t}(x)v\|\leqslant ce^{-\lambda t}\|v\|$ for every $v\in E_x^u,\,t\geqslant 0.$
In the special case that $\Lambda=\{p\}$ is a fixed point then the hyperbolic splitting can be read simply as $T_\Lambda N=E^s\oplus E^u$. 
The flow $\Phi_t:N\to N$ is called an \textit{Anosov flow} if $N$ is a hyperbolic set.

A hyperbolic set $\Lambda$ is a \textit{basic set} if it is transitive, the periodic orbits of $\Phi_t|_{\Lambda}$ are dense in $\Lambda$
and there exists an open set $U\supseteq \Lambda$ with $\Lambda=\bigcap_{t\in\mathbb{R}}\Phi_t(U)$. 
A basic set $\Lambda$ is an \emph{attractor} if there exists an open set $U\supseteq \Lambda$ so that
 $\Lambda=\bigcap_{t>0}\Phi_t(U)$.
The flow $(\Phi_t)_{t\in \mathbb R}$ is called \textit{Axiom A} if its non-wandering set is a hyperbolic set and coincides with the closure
of the set of critical elements, ie. the set formed by periodic orbits and singularities
(recall that the {non-wandering set} of the flow consists of all points $x\in N$ so that $\Phi_t(V) \cap V\neq \emptyset$ for every open neighborhood $V$ of $x$ and every large $t$). 
The spectral decomposition theorem ensures that the non-wandering set of an Axiom A flow consists of a finite number of disjoint hyperbolic basic sets.

Given a hyperbolic set $\Lambda\subset N$, the stable manifold theorem guarantees that there exists $\vep>0$ so that 
the \emph{local strong stable manifold} and \emph{local strong unstable manifold} at a point $x\in \Lambda$ defined, respectively, by
\begin{eqnarray*}
W^s_{\varepsilon}=\Big\{y\in N: d(\Phi_t(x),\Phi_t(y))\leqslant\varepsilon, \;\; \forall t\geqslant 0\;
\text{and}\, \lim_{t\to+\infty}d(\Phi_t(x),\Phi_t(y))= 0 \Big\}
\end{eqnarray*}
and
\begin{eqnarray*}
W^u_{\varepsilon}=\{y\in N: d(\Phi_{-t}(x),\Phi_{-t}(y))\leqslant\varepsilon\,\;\; \forall t\geqslant 0\;
\text{and}\, \lim_{t\to+\infty}d(\Phi_{-t}(x),\Phi_{-t}(y))= 0 \Big\}
\end{eqnarray*}
are embedded submanifols of $N$ invariant by the flow: for each $x\in \Lambda$ and all $t\geqslant 0$ it holds that $\Phi_t(W_{\varepsilon}^s(x))\subset W_{\varepsilon}^s(\Phi_t(x))$ and $\Phi_{-t}(W_{\varepsilon}^u(x)\subset W_{\varepsilon}^u(\Phi_{-t}(x))$.
The \textit{local center-stable} and \textit{local center-unstable} manifolds at a point $x\in \Lambda$ are defined as $W^{cs}_{\varepsilon}(x):=\bigcup_{|t|\leqslant\varepsilon} \Phi_t(W^s_{\varepsilon}(x))$ and $W^{cu}_{\varepsilon}(x):=\bigcup_{|t|\leqslant\varepsilon} \Phi_{-t}(W^u_{\varepsilon}(x))$, respectively.

\subsubsection{Markov partitions and semiconjugacy to a suspension flow}

In the present subsection we recall the symbolic dynamics of Axiom A basic sets, established by Bowen~\cite{Bo73}  and 
Ratner~\cite{Rat} as a byproduct of the
construction of finite Markov partitions.

A differentiable closed $(d-1)$-dimensional disk $\mathcal{D}\subset N$ is called a local \textit{cross-section} if it is transverse to the flow direction. Also, a set $\mathcal{R}\subset\mathcal{D}\cap\Lambda$ is said a \emph{rectangle} if $W^{cs}_{\varepsilon}(x)\cap W^{cu}_{\varepsilon}(y)\cap\mathcal{R}$ consists of exactly one point for $x,y\in\mathcal{R}$.
A finite collection $\mathcal{M}=\{\mathcal{R}_1,\ldots,\mathcal{R}_n\}$ is called a \emph{proper family of size} $\varepsilon$ if
\begin{enumerate}
\item[(i)] each $\mathcal{R}_i$ is a closed subset of $\Lambda$
\item[(ii)] $\Lambda=\bigcup_{t\in[-\varepsilon,0]}\Phi_t(\bigcup_i\mathcal{R}_i)$, and there are local sections $\mathcal{D}_1,\ldots,\mathcal{D}_n$ 
of diameter less than $\vep$ satisfying:
\begin{enumerate}
\item[$\bullet$] $\mathcal{R}_i\subset \text{int}(\mathcal{D}_i)$ and $\mathcal{R}_i=\overline{\text{int}(\mathcal{R}_i})$,
\item[$\bullet$] if $i\neq j$, at least one of the sets $\mathcal{D}_i\cap\Phi_{[0,\varepsilon]}(\mathcal{D}_j)$ and $\mathcal{D}_j\cap\Phi_{[0,\varepsilon]}(\mathcal{D}_i)$ is empty.
\end{enumerate}
\end{enumerate}

Let $\mathcal{M}=\{\mathcal{R}_1,\ldots,\mathcal{R}_n\}$ be a proper family of size $\varepsilon>0$ and set $\Gamma:=\bigcup^n_{i=1}\mathcal{R}_i$. Item (ii) implies that for any $x\in\Phi_{[-\varepsilon,0]}(\bigcup_i\mathcal{R}_i)$ there exists a first positive time $0\leqslant r(x)\leqslant\varepsilon$ so that $\Phi_{r(x)}(x)\in\Phi_{[-\varepsilon,0]}(\bigcup_i\mathcal{R}_i)$. Let $P$ denote the Poincar\'e return map to $\Gamma$ related to $\Phi_t$, and the return time $r$. These maps are continuous on 
\begin{eqnarray*}
\Gamma'=\Big\{ x\in\Gamma\,:\,P^k(x)\in\bigcup_i\text{int}(\mathcal{R}_i),\,\,\forall k\in\mathbb{Z}\Big\}.
\end{eqnarray*}
Finally, a proper family $\mathcal{M}=\{\mathcal{R}_1,\ldots,\mathcal{R}_n\}$ is called \emph{Markov family} if
\begin{enumerate}
\item[(iii)] $x\in U(\mathcal{R}_i,\mathcal{R}_j):=\overline{\{w\in\Gamma'\,:\,w\in\mathcal{R}_i\, ,\,P(w)\in\mathcal{R}_j\}}$ implies that $\mathcal{D}_i\cap W^{cs}_{\varepsilon}(x)\subset U(\mathcal{R}_i,\mathcal{R}_j)$,
\item[(iv)] $y\in U(\mathcal{R}_k,\mathcal{R}_i):=\overline{\{w\in\Gamma'\,:\,w\in\mathcal{R}_i\, ,\,P^{-1}(w)\in\mathcal{R}_k\}}$ implies that $\mathcal{D}_i\cap W^{cu}_{\varepsilon}(y)\subset U(\mathcal{R}_k,\mathcal{R}_i)$.
\end{enumerate}

Defining a transition matrix $A\in {\mathcal M}_{n\times n}(\{0,1\})$ by $A_{i,j}=1$ if and only if there exists $x\in\Gamma'$ so that $x\in\mathcal{R}_i$ and $P(x)\in\mathcal{R}_j$, in \cite{Bo73} Bowen proved that the flow restricted to the hyperbolic basic set $\Lambda$ is semiconjugated
to a suspension semiflow over the subshift of finite type
 $\sigma_A: \Sigma_A\to\Sigma_A$ defined by $\sigma_A(\overline{x})=(x_{i+1})_{i\in\mathbb{Z}}$, where
$
\Sigma_A=\big\{\overline{x}=(x_i)_{i\in\mathbb{Z}}\in \{1,\ldots,n\}^{\mathbb{Z}}\,:\,A_{x_ix_{i+1}}=1\,\forall i\in\mathbb{Z}\big\}.
$
More precisely:

\begin{theorem}\cite{Bo73}\label{thm:BR75}
Let $\Lambda$ be a hyperbolic basic set for a $C^1$-smooth flow $(\Phi_t)_t$. Then there is an matrix $A,$ a topologically mixing subshift of finite type $\sigma_A: \Sigma_A\to\Sigma_A$, a roof function $r:\Sigma_A \to \mathbb R$ and a continuous surjection $\pi:\Sigma_A^{r}\to \Lambda$ such that $\Phi_t\circ \pi=\pi\circ X_t$.
\end{theorem}

\subsubsection{Thermodynamic formalism}\label{subsec:termo}

Let us recall some preliminaries from the thermodynamic formalism of hyperbolic flows. Assume that $\Lambda\subset N$ is a hyperbolic basic set for the flow 
$(\Phi_t)_t$, that $\{\mathcal{R}_i\}_{1\leqslant i \leqslant n}$ is a Markov family and set $\Gamma:=\bigcup^n_{i=1}\mathcal{R}_i$.
Given $\varepsilon>0$ and $T>0$, a subset $E\subset \Lambda$ is called $(T,\vep)-$\emph{separated} if given $x,y\in E$ there exist $t\in [0,T]$ so that $d(\Phi_t(x),\Phi_t(y))\geqslant \varepsilon$.
Given a continuous function $\phi:\Lambda\to\mathbb{R}$, its \emph{topological pressure} is defined by
\begin{eqnarray*}
P_\Lambda((\Phi_t)_t,\phi)=\lim_{\varepsilon\to 0}\limsup_{T\to\infty}\frac{1}{T}\log Z_T(\Phi_t, \phi, \varepsilon),
\end{eqnarray*}
where
$
Z_T((\Phi_t)_t,\phi, \varepsilon)=\sup\Big\{\sum_{x\in E}e^{\int^T_0 \phi(\Phi_tx)dt}:\,E\,\, \text{is a}\,\, (T,\vep)\text{-separated set}\Big\}.
$
By uniform continuity, the pressure function $P_\Lambda((\Phi_t)_t,\phi)$ coincides with the pressure function of $\phi$ with respect to the time-1 map $\Phi_1$. 
Moreover, using the variational principle for continuous maps on a compact metric spaces (see e.g. \cite{BR75}) one has that 
\begin{eqnarray*}
P_\Lambda((\Phi_t)_t,\phi) 
	= P_\Lambda(\Phi_1,\phi)
	=\sup\left\{h_{\nu}(\Phi_1)+\int\phi\, d\nu : \nu\in\mathcal{M}(\Phi_1) \right\},
\end{eqnarray*}
where the set $\mathcal{M}(\Phi_1)$ denotes the space of $\Phi_1$-invariant probabilities.
If, in addition, the potential $\phi: \Lambda \to \mathbb R$ is H\"older continuous then it induces a piecewise H\"older continuous potential
$$
\bar \phi: \Lambda \cap \Gamma \to\mathbb R
	\quad\text{given by} \quad 
	\bar \phi(x)=\int_0^{r(x)} \phi( \Phi_s(x) )\, ds
$$
on $\Lambda\cap \Gamma$ and, ultimately, induces a potential which are locally H\"older continuous in the symbolic metric 
(we refer the reader to \cite{Bo75} for the definition and more details). 
 Using Theorem~\ref{thm:BR75}, Bowen and Ruelle~\cite{BR75} reduced the construction of equilibrium states for the flow to the construction of equilibrium states for subshifts of finite type. The following instrumental lemma bridges between the discrete and continuous time dynamical systems.

\begin{proposition}\label{le:reductionBR75}
Let $(X_t)_t$ be a suspension semiflow on $\Sigma_A^r$ over a subshift of finite type $\sigma_A$, let $\psi : \Sigma_A^r \to\mathbb R$ be a locally H\"older continuous potential and let $\bar\psi: \Sigma_A \to\mathbb R$ defined by $\bar\psi(x)=\int_0^{r(x)} [\psi( X_s(x) )- P_{ \Sigma_A^r}((X_t)_t, \psi)]\, ds$. The following properties are equivalent:
\begin{enumerate}
\item there exists an equilibrium state  $\mu$ for the flow $(X_t)_t$ with respect to $\psi$; 
\item $P_{ \Sigma_A}(\sigma, \bar\psi)=0$ and  there exists an equilibrium state  $\nu$ for $\sigma$ with respect to $\bar\psi$, 
\end{enumerate}
\end{proposition}

In the eventuality that the equilibrium states $\mu$ and $\nu$ above exist and are unique, then
$
\mu=\frac{\nu\times \text{Leb}}{\int_{\Sigma_A} r d\nu }
$
is the lift of the $\sigma_A$-invariant probability $\nu$.

\subsection{Suspension flows of hyperbolic skew-products}\label{subsec:susp}

We begin by describing the abstract framework of suspension (semi)flows over piecewise expanding and piecewise hyperbolic maps, which will be used later to model certain hyperbolic sets for flows.
Given a subset $J\subset\mathbb{R}^d$, we say that $J$ is an \textit{almost John} domain if there exist constants $C_1, \varepsilon_0>0$ and $u\geqslant 1$ such that for all $\varepsilon\in(0,\varepsilon_0)$ and $x\in J$ there exists $y\in J$ so that
 $d(x,y)\leqslant \varepsilon$ and the open ball of center in $y$ and radius $C_1\varepsilon^u$ is contained in $J$.
Throughout the text we will assume that $J$ is almost John and that the boundary of $J$ has {box-counting dimension} strictly less than $d$.
Recall that if $S$ is a non-empty bounded subset of $\mathbb{R}^d$ and let $N_{\delta}(S)$ be the smallest number of sets of diameter at most $\delta$ needed to cover $S$, the \textit{box-counting dimension} of $S$ is defined by 
$\dim_B S=\lim_{\delta\to 0}\frac{\log N_{\delta}(S)}{-\log\delta}$.

\medskip
We say that $T:J\to J$ is a \textit{piecewise expanding $\mathcal{C}^{1+\alpha}$ Markov map} if there exists a finite collection $\mathcal{P}=\{P_1,\ldots,P_m\}$ 
formed by connected open subsets of $J$ satisfying:
\begin{itemize}
\item $\bigcup_{i=1}^m P_i$ is dense in $J$;
\item $T$ is a ${C}^1$ diffeomorphism from $P_i$ to $T(P_i)$ for each $i=1,\ldots,m$;
\item  for each $1\leqslant i \leqslant m$ the image $T(P_i)$ is the union of a subcolection of elements of $\cP$;
\item there exists $\lambda>1$ such that
\begin{eqnarray}\label{h-1}
\|(DT^n(x))^{-1}\|\leqslant \lambda ^{-n}\,\, \text{for all}\,\, x\in J\,\, \text{and}\,\, n\in\mathbb{N},
\end{eqnarray}
\end{itemize}
We may assume that 
$\lambda=\min_{x\in J} \|DT(x)^{-1}\|^{-1}$, and take $\rho:=\max_{x\in J} \|DT(x)\|$. 
Let $\mathcal{P}=\{{P}_1,\ldots,{P}_m\}$ be a finite Markov partition for a $C^ {1+\alpha}$ piecewise expanding map $T$ of the interval $[0,1]$ such that $T\mid_{P_i}$ is injective, for every $1\leqslant i \leqslant m$. 
For each $n\geqslant 1$ let $\mathcal{P}^{(n)}$ be the refined partition 
$
\bigvee^{n-1}_{j=0} T^{-j}(\mathcal{P})=
	\big\{{P}_{i_0}\cap T^{-1}(P_{i_1})\cap\cdots\cap T^{-(n-1)}({P}_{i_{n-1}})\,:\,{P}_{i_j}\in\mathcal{P}\big\},
$
and let $\mathcal{P}^{(n)}(x)$ 
denote the partition element of $\mathcal{P}^{(n)}$ containing $x$.
The Markov property ensures that
$T(\mathcal{P}^{(n)}(x))
	= \mathcal{P}^{(n-1)}(T(x)).
	$
Given $n\geqslant 1$ we denote by $\mathcal H_n$ the collection of inverse branches for $T^n\mid_P$, with $P\in \mathcal P^{(n)}$. 
We will also require the following:
\begin{itemize}
\item there exists $C_2>0$ such that, for every $n\geqslant 1$,
\begin{eqnarray}\label{h0}
\| Dh_n(x) - Dh_n(y)\| \leqslant C_2 \, \|Dh_n(x)\|\, d(x,y)^{\alpha}\,\, \text{for all} \; h_n \in \mathcal H_n\; \text{and}\;  x,y\in \mbox{Dom}(h_n).
\end{eqnarray}
\end{itemize}
The previous  H\"older continuity property of the derivative of inverse branches (which holds in dimension one as a simple consequence that $\log \|DT\|$ is H\"older continuous) will be used solely to spread a pointwise UNI condition to uniform size
neighborhoods around those points (see the proof of Lemma~\ref{le:UNI} and Proposition~\ref{prop:transv-implies-wUNI}).

Finally we require that $T$ is \textit{covering}, meaning that for every open ball $B\subset J$ there exists $n\in\mathbb{N}$ such that $T^nB=J$. 
In the present context this is not a strong requirement in comparison to transitivity. On the one hand, any transitive piecewise expanding Markov map is semiconjugate to a transitive subshift of finite type. On the other hand, 
the spectral decomposition 
ensures that there exists a finite union $\tilde J\subset J$ of
elements of $\cP$ and $N\geqslant 1$ so that $T^N\mid_{\tilde J}$ is 
covering.  Hence, 
reducing the cross-section if necessary, 
a suspension semiflow over a transitive piecewise expanding Markov map
can be modeled by a piecewise expanding, Markov and covering map.
Finally, recall that any such map $T$ is called \emph{full branch} if $T(P_j)\supset \bigcup_{i=1}^m P_i$ for every $1\leqslant j \leqslant m$.

\medskip
Given a compact Riemannian manifold $M$, let $d$ denote the distance induced by the Riemannian metric. We say that a skew-product $F: J \times M \to J \times M$ given by $F(x,y)=(T(x),G(x,y))$
is a $C^{1+\al}$ \emph{piecewise hyperbolic skew-product} if:
\begin{itemize}
\item
 $T:J\to J$ is a ${C}^{1+\alpha}$-piecewise expanding  Markov map,
 \item there are constants $C_3>0$ and $\gamma_0\in (0,1)$ such that 
$$
d(F^n(x,y),F^n(x,z))\leqslant C_3{\gamma_0}^n d(y,z) \quad \text{for each $x\in J$ 
and all $y,z\in M$.}
$$
\end{itemize}
Observe that the natural projection 
$
\pi:J\times M\to J
	\;\text{given by}
	\;
	\pi(x,y)=x
$ 
establishes a semiconjugacy between the skew-product $F$ and the piecewise expanding
map $T$.

\medskip
Given a metric space $Y$, a function $f: Y \to Y$ and a roof function $r: Y \to \mathbb{R}^+$ so that $\inf r>0$, 
the \textit{suspension semiflow} $(X_t)_{t\geqslant 0}$ (determined by $f$ and $r$) 
evolves on 
$$
Y^r=\{(x,t)\in Y\times \mathbb{R}^{+}:\, 0\leqslant t\leqslant r(x)\}/ \sim,
$$
where the quotient space is characterized by the equivalence relation $(x, r(x))\sim (f(x), 0)$ $\forall x\in Y$. In local coordinates, the suspension semiflow $(X_t)_{t\geqslant 0}$ is given by
\begin{eqnarray}\label{def:semiflow}
X_t(x, s)=\Big(f^n(x), t+s-\sum^{n-1}_{i=0} r(f^i(x))\Big)
	\quad \text{ for all $t\geqslant 0$},
\end{eqnarray}
where the integer $n=n(x,s,t)\geqslant 0$ is 
determined by 
$\sum^{n-1}_{i=0} r(f^i(x))\leqslant s+t<\sum^{n}_{i=0} r(f^i(x))$.
If $f$ is invertible then the expression \eqref{def:semiflow} makes sense for all $t \in \mathbb R$ and $(X_t)_{t\in \mathbb R}$ defines a flow. Moreover, given an $f$-invariant, ergodic probability $\mu$, the $(\Phi_t)_t$-invariant probability measure 
defined by 
\begin{equation}\label{eq:susp-measure}
\mu^r:=\dfrac{\mu\ltimes \leb}{\int_Y r \,d\mu}
\end{equation}
is ergodic. 
We will consider suspension semiflows over piecewise expanding ${C}^{1+\alpha}$ Markov maps
$(T,J,\mathcal P)$ and roof functions $r$ satisfying the following requirements:
\begin{itemize}
\item $r$ is $C^{1+\alpha}$-smooth on each Markov domain; 
\item there exists $C_4>0$ so that
\begin{equation}\tag{H1}\label{h1}
\|D(r\circ h)(x)\|\leqslant C_4\,\,\,\; \text{for all}\,\,x\in \mbox{Dom}(h) \text{ and } h\in \mathcal H
\end{equation}
where $\mathcal H$ denotes the space of inverse branches of $T$; 
\item $r$ is bounded and satisfies 
\begin{equation}\tag{H2}\label{h2}
0<\inf{r} \leqslant r(x)\leqslant 1,\,\,\,\; \text{for all}\,\,\, x\in J.
\end{equation}
\item $r$ is not ${C}^1$-piecewise cohomologous to a locally constant roof function, i.e., there exists no $\theta\in{C}^1(J,\mathbb{R})$ so that $r=\theta\circ T-\theta+\chi$ where $\chi$ is constant on each partition element.

\end{itemize}

\subsection{Adapted partitions for piecewise expanding Markov interval maps}\label{subsec:addapted-partititons} 

The strategy to overcome the absence of the Federer property (which would 
allow to compare measure of adjacent cylinders) is to use the 
Markov and Gibbs properties to compare the measure of
cylinders with certain well chosen sub-cylinders which cover domains on which the transfer operator exhibits cancellations. 
This argument improves the ideas used by the authors in the special context of full branch piecewise expanding maps  \cite{DV}, 
to produce an appropriate partition formed by cylinders of different depths but similar sizes ${|b|}^{-1}$.

Let $T$ be a piecewise expanding interval map and recall that
$\rho=\max_x|T'(x)| \geqslant \min_{x} |T'(x)| = \lambda > 1$. 
Note that the mean value theorem on domains of smoothness for the maps $T^i$ implies that for every 
{$x\in \bigcup_{P\in \cP^{(n)}} P$}, 
\begin{eqnarray}\label{eq:diame}
\rho^{-i}\diam(\mathcal{P}(T^{i}(x)))\leqslant \diam(\mathcal{P}^i(x))\leqslant \lambda^{-i}\diam(\mathcal{P}(T^i(x)))
\quad \text{for each $i\geqslant 0$.}
\end{eqnarray}
The following lemma provides a partition of the interval (non-homogeneous in the number of refinements) by cylinders with essentially the same size. 

\begin{lemma}\label{le:partQ}
If $\Delta>0$ and $|b|> \dfrac{2\Delta\rho}{\min \diam ({P}_i)}$ then
there exists a finite partition $\cQ$ of the interval $[0,1]$ (except for a finite number of points) 
such that for every $Q\in \cQ$ there exists $\ell=\ell(Q,b)\geqslant 1$
such that $Q\in \cP^{(\ell)}$ and 
\begin{equation}\label{est:size}
\dfrac{2\Delta}{|b|}\leqslant \diam(Q) \leqslant \dfrac{2\Delta}{|b|} \rho.
\end{equation}
\end{lemma}

\begin{proof}
While the argument is identical to the one used in \cite[Lemma~3.3]{DV}, we include it for completeness. 
Fix $b$ as above.  Given $x\in I:=[0,1]$, 
let $\ell = \ell(x,b)\geqslant 1$ be the largest integer such that $\diam \cP^{(\ell)}(x) \geqslant 2\Delta |b|^{-1}$.
Such an integer $\ell$ does exist since $\min\diam \mathcal{P}_i > 2\Delta |b|^{-1} \rho$ and the diameter of the cylinders is exponentially decreasing, as a consequence of ~\eqref{eq:diame}. 
Inequalities ~\eqref{eq:diame} also ensure that $\diam(\cP^{(\ell)}(x)) \leqslant 2\Delta |b|^{-1} \rho$, because otherwise $\diam(\cP^{(\ell+1)}(x)) \geqslant 2\Delta |b|^{-1}$. By construction we have that $ \ell(x,b)= \ell(y,b)$ for every $y\in  \cP^{(\ell)}(x)$, hence we may determine that $Q= \cP^{(\ell)}(x)$ belongs to $\cQ$ and set $\ell(Q,b)= \ell(x,b) \geqslant 1$.
The collection of cylinder sets
$
\mathcal{Q}=\big\{\mathcal P^{\ell(x,b)}(x): x\in I \big\}
$
 has at most $(2\De)^{-1}|b|$ elements, which are pairwise disjoint elements and cover $I$ up to a finite number of points, and each element $Q\in \cQ$ satisfies \eqref{est:size}. This proves the lemma.
\end{proof}

As Gibbs measures are fully supported and non-atomic (see e.g. \cite{Bo75}), 
every Gibbs measure gives zero weight to the points in the boundary of elements in $\cQ$. 
In consequence, the finite collection $\cQ$ in Lemma~\ref{le:partQ} is adapted to \emph{all} Gibbs measures. 
The next lemma shows that the inconvenient dependence of $\ell$ on $b$ can be disregarded 
when estimating the measure of sets with size of order $\frac1{|b|}$.

\begin{lemma}\label{le:cancellations-measure}
Fix $0<\delta<\Delta$, let $\mathcal Q=(Q_i)_{1\leqslant i \leqslant \# \cQ}$ be an enumeration of the elements in $\cQ$ and let 
$\mathcal J=(J_i)_{1\leqslant i \leqslant \# \cQ}$ be any collection of intervals of diameter $2\delta {|b|^{-1}}$ such that 
each $J_i \subset Q_i$ 
for every $1\leqslant i \leqslant \# \cQ$. There exists a constant $\gamma>0$ (depending on $\mu$ but independent of $b$) such that
$\mu(J_i) \geqslant \gamma \mu(Q_i)$ for every $1\leqslant i \leqslant \#\cQ$.
\end{lemma}

\begin{proof}
The proof is identical to the one of \cite[Lemma~3.5]{DV}, hence we shall omit it. 
\end{proof}

The previous property will be crucial to estimate cancellations by looking at deeper cylinder sets 
instead of looking at the measure of cylinders sets that are neighbors.
More precisely:

\begin{corollary}\label{comparijpropfinite}
Assume that $\mathcal J=(J_i)_{1\leqslant i \leqslant \# \cQ}$ is any collection of intervals as in 
Lemma~\ref{le:cancellations-measure}. There exists $\delta'>0$ (depending on $\mu$ and independent of $b$) 
so that, if a {piecewise H\"older continuous} $w: I \to \mathbb R^+$  
satisfies $|\log w|_{\alpha}\leqslant K|b|^{\alpha}$  
then
\begin{eqnarray*}
\int_{{J}}w\, d\mu\geqslant \delta'\int_{{I}}w\, d\mu,
\end{eqnarray*}
where ${J}=\bigcup_{1\leqslant i \leqslant \#\cQ} J_i$. 
\end{corollary}

\begin{proof}
It is enough to show there exists a uniform $\delta'>0$ so that 
$
\int_{{J_i}}w\, d\mu\geqslant \delta'\int_{{Q_i}}w\, d\mu$
$\text{for every $1\leqslant i \leqslant \#\cQ$.}
$
Fix $1\leqslant i \leqslant \#\cQ$. By construction $\diam(Q_i) \leqslant {2\Delta}{|b|^{-1}\rho} $ (recall Lemma~\ref{le:partQ}).
This, together with the assumption that $|\log w|_{\alpha}\leqslant K|b|^{\alpha}$  ensures that 
\begin{align}\label{eq:Holderquot}
\Big|\log \frac{w({x})}{w({y})} \Big| 
	 \leqslant K |b|^{\alpha} d(x,y)^\alpha 
	 \leqslant K'
	 \quad \text{for any $x,y\in Q_i$},
\end{align}
where $K'=K\left( 2\Delta \rho \right)^{\alpha}>0$. 
Altogether, Lemma~\ref{le:cancellations-measure} and inequality~\eqref{eq:Holderquot} imply
$
\int_{J_i} w\, d\mu 
	\geqslant \mu(J_i) \,\inf_{J_i} w
	\geqslant 	 \gamma \mu (Q_i) \; e^{-K'} \max_{Q_i} w
	\geqslant 	 \gamma e^{-K'} \, \int_{Q_i} w\, d\mu.
	$
The lemma follows with $\delta'=\gamma e^{-K'}>0$, which is independent of both $b$ and $w$.
\end{proof}

\section{Transfer operators}\label{transfer}
\subsection{Twisted transfer operators}

In this subsection we define the suitable transfer operators and Banach spaces.  Given $\alpha\in (0,1]$, consider the space 
$$
C^{\alpha}(J):=\{v:J\to\mathbb{C}\,:\,|v|_{C^{\alpha}(J)}<\infty\}
$$ 
of $\alpha$-H\"older continuous potentials, where $|v|_{C^{\alpha}(J)}:=\sup_{x\neq y} \frac{|v(x)-v(y)|}{d(x,y)^{\alpha}}$. We remark that as $J$ is the disjoint union of a finite number of connected subsets of $\mathbb{R}^n$, the previous supremum is taken over all $x,y\in J$ which are in the same connected component. It is well known that this is a Banach space when equipped with the norm $\|\cdot\|_{C^{\alpha}(J)}:=|\cdot|_{C^{\alpha}(J)}+\|\cdot\|_{L^{\infty}(J)}$ (here $L^{\infty}(J)$ will always denote the $L^\infty$-space with respect to a certain fixed equilibrium state in the class we will describe below). For each 
$b\in\mathbb{R}$ we consider the equivalent norm 
\begin{equation}\label{eq:def-norm-b}
\|\cdot\|_{(b)}:=\frac{1}{1+|b|^{\alpha}}|\cdot|_{C^{\alpha}(J)}+\|\cdot\|_{L^{\infty}(J)}.
\end{equation}
It is clear that $|\cdot|_{C^{\alpha}(J)} \leqslant (1+|b|^{\alpha}) \; \|\cdot\|_{(b)}$ for every $b\in \mathbb R$.
\smallskip

Let $\mathcal{P}=\{{P}_1,\ldots,{P}_m\}$ be the finite Markov partition associated to the $C^{1+\alpha}$ piecewise expanding map $T$ defined on $J$. 
For each $s=\sigma+ib\in\mathbb{C}$ and each piecewise H\"older continuous potential $\phi$ one can 
associate the \textit{twisted transfer operator} $P_s: C^\alpha(J)\to C^\alpha(J)$ defined as
\begin{eqnarray}\label{def:P-twisted}
(P_s v)(x) 
	=\sum_{y\in T^{-1}(x)}e^{(\phi-s r)(y)}v(y), \quad \forall x\in J. 
\end{eqnarray}
Writing the Birkhoff sums $S_n (\phi-sr)(x)=\sum_{i=0}^{n-1}(\phi-sr)(T^i(x))$ we obtain, inductively, 
\begin{eqnarray*}
(P_s^nv)(x)=\sum_{y\in T^{-n}(x)}e^{S_n (\phi-sr)(y)}v(y)
	=\sum_{\omega\in\mathcal{P}^{(n)}}(e^{S_n(\phi-sr)}v)\circ h_{\omega}\cdot \mathbf{1}_{T^n\omega}, 
	\quad \forall n\geqslant 1.
\end{eqnarray*}

\begin{remark}\label{perron}
\emph{
The spectral properties for the twisted transfer operators $P_{\sigma}:C^\alpha(J)\to C^\alpha(J)$ rely on Ruelle's theorem 
(cf. \cite[Theorem~2]{Pol84} and  \cite[Proposition~1]{Pol85}): for each $\sigma\in\mathbb{R}$ there exist $\lambda_{\sigma}>0$, a probability measure $\nu_{\sigma}$ and a real valued positive eigenfunction $f_{\sigma}\in C^\alpha(J)$ which is bounded away from zero so that: (i)  $P_{\sigma}f_{\sigma}=\lambda_{\sigma} f_{\sigma}$; (ii) $P^{*}_{\sigma}\nu=\lambda_{\sigma}\nu_{\sigma}$, where $P_{\sigma}^{*}$ represents the dual operator of $P_{\sigma}$; 
and (iii) $\mu_{\sigma}:=f_{\sigma}\nu_{\sigma}$ is the unique equilibrium state for $T$ with respect to the potential 
$\phi-\sigma r$.
Moreover, both probabilities $\nu_{\sigma}$ and $\mu_\sigma$ are Gibbs measures and are non-singular with respect to $T$, and their Jacobians are $J_{\nu_{\sigma}}T=\lambda_{\sigma} e^{-(\phi-\sigma r)}$ and $J_{\mu_{\sigma}}T=\lambda_{\sigma} e^{-(\phi-\sigma r)}\dfrac{f_{\sigma}\circ T}{f_{\sigma}}$, respectively.
}
\end{remark}

We may define now the family of \textit{normalized transfer operators}
\begin{eqnarray}\label{def:P-twisted-normalized}
 {\mathscr L}_sv=(\lambda_{\sigma}f_{\sigma})^{-1}P_s(f_{\sigma}v),
\end{eqnarray}
associated to parameters $s=\sigma+ib\in\mathbb{C}$.
Observe that ${\mathscr L}_\sigma 1=1$ for all $\sigma\in \mathbb R$ and $|{\mathscr L}_s|_{L^{\infty}(\mu_{\sigma})}
\leqslant 1$ for all such values of $s$. 
It will be useful to write the iterates of ${\mathscr L}_s$ as
\begin{eqnarray*}
({\mathscr L}^n_sv)(x)=\lambda_{\sigma}^{-n}f^{-1}_{\sigma}(x)\, \sum_{\omega\in \mathcal{P}^{(n)}} e^{S_n (\phi-sr)(h_{\omega}(x))}(f_{\sigma}v)(h_{\omega}(x)),
\end{eqnarray*}
where $h_{\omega}:=(T^n|_{\omega})^{-1}$ represents the inverse branch of $T^n$ on $\omega$. Furthermore we can relate iterates of both operators by $P^n_sv=\lambda_{\sigma}^nf_{\sigma}{\mathscr L}^n_s(f^{-1}_{\sigma}v)$, for every $n\geqslant 1$.

\begin{remark}\label{pertubationop}
\emph{
Given $s=\sigma+ib$, the classical perturbation theory ensures that one may find $\varepsilon\in(0,1)$ so that the family $s\mapsto P_s$ of operators on $C^{\alpha}(J)$ is continuous on $\{\sigma>-\varepsilon\}$, and $\sup_{|\sigma|<\varepsilon} \|P_s\|_b<\infty$ (see e.g. \cite[Proposition 2.5]{AM}). Reducing $\vep$, if necessary, we may assume $1\leqslant|f_{\sigma}|_{\infty}|f^{-1}_{\sigma}|_{\infty}\leqslant 4|f_0|_{\infty}|\frac1{f_0}|_{\infty}<\infty$ and $|f_{\sigma}|_{\alpha}|f^{-1}_{\sigma}|_{\infty}\leqslant 4|f_0|_{\alpha}|\frac1{f_0}|_{\infty}<\infty$. By Ruelle's theorem there exists a simple leading eigenvalue 
$\lambda_0>0$ 
with strictly positive $C^{\alpha}$ eigenfunction $f_0$ associated to the unperturbed transfer operator $P_0$. In particular, there exists 
$\varepsilon>0$ such that the maps $(-\vep,\vep) \ni \sigma \mapsto \lambda_{\sigma}$ and  $(-\vep,\vep) \ni \sigma \mapsto f_{\sigma}$ are continuous.
}
\end{remark}

In the remaining of this subsection we will provide estimates on the behavior of these twisted transfer operators, including 
a Lasota-Yorke inequality. 

\begin{lemma}\label{l-hypothesis}
For any $\sigma_0>0$ there exists $C_6>0$ so that
$
\lambda_{\sigma}^{-n} \sum_{\omega\in \mathcal{P}^{(n)}} \|e^{S_n(\phi-\sigma r)}\|_{L^{\infty}(\mu_{\sigma}|_{\omega})} \leqslant C_6,
$
for every $|\sigma|\leqslant \sigma_0$. 
In particular, for all $n\in\mathbb{N},\,v\in L^{\infty}(\mu_\sigma)$ and $s=\sigma+it$ with $|\sigma|\leqslant \sigma_0$,
$$
\|\mathscr{L}^n_sv\|_{L^{\infty}(\mu_\sigma)}\le C_6 \|v\|_{L^{\infty}(\mu_\sigma)}
\quad\text{and}\quad
\|P^n_sv\|_{L^{\infty}(\mu_\sigma)}\le C_6 \lambda_\sigma^n \|v\|_{L^{\infty}(\mu_\sigma)}.
$$
\end{lemma}

\begin{proof}
Using that $\mathbf 1={\mathscr L}_\sigma^n \mathbf 1=\lambda_\sigma^{-n} \frac{1}{f_\sigma}\sum_{\omega\in\cP^{(n)}}(e^{S_n(\phi-\sigma r)}f_\sigma)\circ h_{\omega}$ we get
\begin{eqnarray*}
0\leqslant \lambda_\sigma^{-n} \sum_{\omega\in\cP^{(n)}}e^{S_n(\phi-\sigma r)(h_{\omega}x)}
	&=& \lambda_\sigma^{-n}  \frac{f_\sigma(x)}{f_\sigma(x)}\sum_{\omega\in\cP^{(n)}} e^{S_n(\phi-\sigma r)(h_{\omega}x)}\frac{f_\sigma(h_{\omega}x)}{f_\sigma(h_{\omega}x)}\\
	&\leqslant & |f_\sigma|_{\infty}\Big|\frac1{f_\sigma}\Big|_{\infty}
	{\mathscr L}_\sigma^n \mathbf 1.	
\end{eqnarray*}
Hence, the right-hand side above is bounded by $\max_{\sigma\in [-\sigma_0,\sigma_0]} |f_\sigma|_{\infty}\,|\frac1{f_\sigma}|_{\infty}$, for every $x\in J$
and every $|\sigma|\leqslant \sigma_0$. 
Moreover, using the uniform backward contraction for $T$, we have the following bounded distortion estimate: $|S_n(\phi-\sigma r)(h_{\omega}x) -S_n(\phi-\sigma r)(h_{\omega}y)| \leqslant \frac{1}{1-\lambda^{-\alpha}} \, |\phi-\sigma r|_\alpha\,\diam(J)^\alpha$ for every $x,y$ in the domain of $h_\omega$.
Together with the previous estimate, this implies 
\begin{align*}
\lambda_{\sigma}^{-n} \sum_{\omega\in \mathcal{P}^{(n)}} \|e^{S_n(\phi-\sigma r)}\|_{L^{\infty}(\mu_{\sigma}|_{\omega})} 
	& \leqslant e^{\frac{1}{1-\lambda^{-\alpha}} \, |\phi-\sigma r|_\alpha\,\diam(J)^\alpha} 
		\lambda_{\sigma}^{-n}  \sum_{\omega\in\cP^{(n)}}e^{S_n(\phi-\sigma r)(h_{\omega}x)} 
	 \leqslant C_6
\end{align*}
where $C_6:=\max_{|\sigma|\leqslant \sigma_0} e^{\frac{1}{1-\lambda^{-\alpha}} \, |\phi-\sigma r|_\alpha\,\diam(J)^\alpha} \, |f_\sigma|_{\infty}\Big|\frac1{f_\sigma}\Big|_{\infty}$.
As the second claim in the lemma is an easy consequence of the first one, this proves the lemma. 
\end{proof}

\begin{lemma}\label{lemma1}
There exists {$C_7 \geqslant (1-\lambda^{-1})C_4$} such that $|D(S_nr\circ h_{\omega})(x)| \leqslant \frac{1}{2}C_7$ for all $n\in\mathbb{N},\,\omega\in\mathcal{P}^{(n)}$ and $x\in T^n\omega$.
\end{lemma}

\begin{proof}
Fix $n\in\mathbb{N},\,\omega\in\mathcal{P}^{(n)}$ and $x\in T^n\omega$, and
set $y=h_{\omega}(x)$, where $h_\omega$ is the inverse branch of $T^n\mid_\omega$. By the chain rule,
$
D(S_nr\circ h_{\omega})(x)=\sum^{n-1}_{k=0}Dr(T^ky)D(T^k\circ h_{\omega})(T^ny).
$
As one can write
\begin{eqnarray*}
Dr(T^ky)D(T^k\circ h_{\omega})(T^{n}y)=Dr(T^ky)DT(T^ky)^{-1}DT(T^ky)D(T^k\circ h_{\omega})(T^{n}y)
\end{eqnarray*}
for each $0\leqslant k \leqslant n-1$, assumption \eqref{h1}, inequality \eqref{h-1} and the fact that $\|DT(z)\|\leqslant \rho$ for every $z\in J$ ensure that
$
|D(S_nr\circ h_{\omega})(x)\|\leqslant C_4{\rho}\sum^{n-1}_{k=0}\lambda^{-(n-k)}\leqslant \frac{1}{2}C_7,
$
where the constant {$C_7:=\max\{2C_4{{\rho}}{(1-\lambda^{-1})^{-1}}, (1-\lambda^{-1})C_4\} >0$ } is independent of $n, \omega$ and $x$. 
\end{proof}

\subsection{Lasota-Yorke inequality}
We will make use of the following instrumental estimates.

\begin{proposition}[Lasota-Yorke inequality]\label{L-Y0}
There exists $\sigma_0>0$ and $C_8>0$ so that, for each $s=\sigma+ib,\,|\sigma|<\sigma_0,\,v\in {C}^{\alpha}(J)$ and $n\in\mathbb{N}$, the following properties hold:
\begin{enumerate}
\item\label{L-Y1} $
	|{\mathscr L}_s^nv|_{{C}^{\alpha}(J)}\leqslant  
	C_8 \lambda^{-\alpha n} |v|_{{C}^{\alpha}(\mu_\sigma)} + C_8 (1+|b|^{\alpha})\|v\|_{L^{\infty}(\mu_\sigma)}$;
\item\label{L-Y2} $ 
	|{\mathscr L}_s^nv|_{(b)}\leqslant  
	C_8 \lambda^{-\alpha n}\|v\|_{(b)}+ C_8 \|v\|_{L^{\infty}(\mu_\sigma)}$.
\end{enumerate}
\end{proposition}

\begin{proof}
Fix $n\geqslant 1$, $\omega\in\cP^{(n)}, v\in{C}^{\alpha}(J)$ and $x, y\in T^n\omega, x\neq y$. We write
\begin{align}
(e^{S_n(\phi-sr)}v)(h_{\omega}x)-(e^{S_n(\phi-sr)}v)(h_{\omega}y)
&=(e^{-ibS_nr(h_{\omega}x)}-e^{-ibS_nr(h_{\omega}y)})(e^{S_n(\phi-\sigma r)}v)(h_{\omega}x) \label{LY1}\\
&+e^{-ibS_nr(h_{\omega}y)}(e^{-\sigma S_nr(h_{\omega}x)}-e^{-\sigma S_nr(h_{\omega}y)})(ve^{S_n\phi})(h_{\omega}x)
	\label{LY2}\\
&+e^{-sS_nr(h_{\omega}y)} [v(h_{\omega}x)-v(h_{\omega}y)] e^{S_n\phi(h_{\omega}x)} \label{LY3}\\
&+e^{-sS_nr(h_{\omega}y)}v(h_{\omega}y)(e^{S_n\phi(h_{\omega}x)}-e^{S_n\phi(h_{\omega}y)}).\label{LY4}
\end{align}
Using the inequality $|e^{it}-1|\leqslant 2\min\{1, |t|\}\leqslant 2|t|^{\alpha}$ valid for all $t\in\mathbb{R}$ together with 
Lemma~\ref{lemma1}, we can bound \eqref{LY1} as
\begin{align*}
|(e^{-ibS_nr(h_{\omega}x)} &-e^{-ibS_nr(h_{\omega}y)}) \; (e^{S_n(\phi-\sigma r)}v)(h_{\omega}x)| \\
	&\leqslant 2\min\big\{1, |b| \, |S_nr(h_{\omega}x)-S_nr(h_{\omega}y)|\big\} \; (e^{S_n(\phi-\sigma r)}|v|)(h_{\omega}x)\\
	&\leqslant 2\min\big\{1, |b|\frac{C_7}{2}d(x,y)\big\}(e^{S_n(\phi-\sigma r)}|v|)(h_{\omega}x)\\
	&\leqslant 2^{1-\alpha} C_7^\alpha \,|b|^{\alpha} \, (e^{S_n(\phi-\sigma r)}|v|)(h_{\omega}x) \, d(x,y)^{\alpha}.
\end{align*}
Similarly, as $|e^{-\sigma t}-1| \leqslant 2|\sigma| t$ for every $t>0$,
the term \eqref{LY2} 
is bounded by 
\begin{eqnarray*}
|(e^{-\sigma S_nr(h_{\omega}x)}-e^{-\sigma S_nr(h_{\omega}y)})(ve^{S_n\phi})(h_{\omega}x)|
	&= &(e^{S_n(\phi-\sigma r)}|v|)(h_{\omega}x)|e^{-\sigma (S_nr(h_{\omega}x)-S_nr(h_{\omega}y))}-1|\\
	&\leqslant &C_7 |\sigma|\, (e^{S_n(\phi-\sigma r)}|v|)(h_{\omega}x) \, d(x,y).
\end{eqnarray*}
Concerning the term \eqref{LY3}, using $\|(DT^n(x))^{-1}\|\leqslant \lambda ^{-n}$ for all $x\in J$ and $n\in\mathbb{N}$, that $\sigma > -\sigma_0$ and the previous ideas, it is simple to check it is bounded above by
\begin{align*}
e^{S_n(\phi-\sigma r)(h_{\omega}x)}\, & |e^{-\sigma [S_nr(h_{\omega}y)-S_nr(h_{\omega}x)]}|\, \lambda^{-\alpha n} |v|_{{C}^{\alpha}(J)} d(x,y)^\alpha \\
	& \leqslant \lambda^{-\alpha n}  |v|_{{C}^{\alpha}(J)} {e^{\sigma_0\frac{C_7}{2}\diam(J)}} \, (e^{S_n(\phi-\sigma r)})(h_{\omega}x)\,d(x,y)^{\alpha}.
\end{align*}
Finally, we can estimate \eqref{LY4} along the same lines, obtaining
\begin{eqnarray*}
|e^{-sS_nr(h_{\omega}y)}v(h_{\omega}y)(e^{S_n\phi(h_{\omega}x)}-e^{S_n\phi(h_{\omega}y)})|&\leqslant & |S_n\phi(h_{\omega}x)-S_n\phi(h_{\omega}y)|(e^{S_n(\phi-\sigma r)}|v|)(h_{\omega}y)\\
&\leqslant& \frac{1}{1-\lambda^{-\alpha}} |\phi|_{{C}^{\alpha}(J)} \, (e^{S_n(\phi-\sigma r)}|v|)(h_{\omega}y)\, d(x,y)^{\alpha}.
\end{eqnarray*}
Summing over all $\omega\in\cP^{n}$ we conclude that 
$\dfrac{|{\mathscr L}_s^n v(x)-{\mathscr L}_s^n v(y)|}{d(x,y)^{\alpha}}$ is bounded by
\begin{eqnarray*} 
	\lambda_\sigma^{-n} 
\sum_{\omega\in \mathcal{P}^{(n)}} \|e^{S_n(\phi-\sigma r)}\|_{L^{\infty}(\mu_{\sigma}|_{\omega})} \cdot 
\left[
	\left(
\frac{2}{2^\alpha}C_7|b|^{\alpha}+{C_7} |\sigma|+\frac{|\phi|_{{C}^{\alpha}(J)}}{1-\lambda^{-\alpha}}
	\right)
\|v\|_{L^{\infty}(\mu_{\sigma})}
	+{e^{\sigma_0\frac{C_7}{2}\diam(J)}} \lambda^{-\alpha n}|v|_{C^{\alpha}(J)}
\right].
\end{eqnarray*}
Lemma~\ref{l-hypothesis} guarantees that the first component in the product
is bounded by a uniform constant $C_6$. Since $x\neq y$ are arbitrary 
in a same connected component of $J$, there exists $C_8>0$ (independent of $b$) so that 
$$
|{\mathscr L}_s^n v|_{C^{\alpha}(J)}
	\leqslant C_8 \lambda^{-\alpha n}|v|_{C^{\alpha}(J)} + C_8(1+|b|^{\alpha}) \, \|v\|_{L^{\infty}(\mu_{\sigma})} 
$$
This proves item \eqref{L-Y1}.

Concerning inequality \eqref{L-Y2}, it follows from the first inequality \eqref{L-Y1}, the definition of the $\|\cdot\|_{(b)}$-norm (recall \eqref{eq:def-norm-b})
and the fact that 
$\|{\mathscr L}_s^n v\|_{L^\infty(\mu_\sigma)} \leqslant C_6 \|v\|_{L^\infty(\mu_\sigma)}$, that 
\begin{align*}
\| {\mathscr L}_s^n v \|_{(b)} 
& =\frac{1}{1+|b|^{\alpha}}|P_s^nv|_{C^{\alpha}(J)}+\|P_s^nv\|_{L^{\infty}(J)}	\\
	& \leqslant C_8 \lambda^{-\alpha n} \frac{1}{1+|b|^{\alpha}} |v|_{C^{\alpha}(J)} + (C_6+C_8) \, \|v\|_{L^{\infty}(\mu_{\sigma})}, 
\end{align*}
which proves the lemma (taking a slightly larger constant $C_8$ than defined in the first item).

\end{proof}

\section{Transversality and uniform nonintegrability condition}\label{sec:UNI-transversal}

After the breakthtrough contribution of Dolgopyat \cite{D} using the uniform non-integrability condition (introduced by Chernov) to prove exponential decay of correlations for hyperbolic flows. 
In  \cite{BW} Butterley and War constructed open sets of codimension one exponentially mixing Anosov flows using a transversality condition, which 
holds whenever the stable and unstable bundles are not jointly integrable and it
is sufficient to ensure the exponential mixing for the SRB measure.  The van der Corput lemma which garantees cancellations in \cite{BW} seems
unlikely to hold for other Gibbs measures. For that reason, after adapting the concept of transversality in order to deal with other Gibbs measures,
we need to obtain cancellations using a uniform nonintegrability condition as in \cite{AM,AGY,BV}. In opposition to these other contexts, the uniform integrability  
condition will be fitted to Markov maps which are not necessarily full branch.

In what follows, let $T$ be a $C^{1+\alpha}$ piecewise expanding map defined on $J$ and $r: J \to \mathbb R_+$ be a roof function.  For each $n\geqslant 1$ let $\mathcal{P}^{(n)}$ be the n-th refinement of the Markov partition $\cP$.

\subsection{Uniform non-integrability and cancellations in uniform scales}

Let us recall the uniform non-integrability condition (UNI) on the roof function, which is used to measure some shear along 
the flow direction. In the special case that $T$ is a piecewise expanding and full branch Markov map one can define 
\begin{eqnarray}\label{eq:def-psi}
\psi_{h_1,h_2}:=S_n r \circ h_1- S_n r\circ h_2:J\to\mathbb{R}
	\quad \text{for each pair $h_1,h_2\in\mathcal{H}_n$,}
\end{eqnarray}
where $S_n r=\sum^{n-1}_{j=0} r\circ T^j$ and $h_j:=h_{\om_j}$ with $\om_j\in \mathcal P^{(n)}$ for $j=1,2$.
More generally, $\mbox{Dom}(\psi_{h_\om,h_{\tilde\om}})= T^n(\om) \cap T^n(\tilde\om)$ 
for all $\om,\tilde\om\in \mathcal P^{(n)}$.

\smallskip
\begin{definition}\label{def:UNI}
\emph{
The suspension semiflow $(X_t)_{t\geqslant 0}$ over a piecewise expanding and full branch Markov map $T$ with roof function $r$ satisfies the 
\emph{uniform non-integrability (UNI) condition}
if there exists $D>0$ and inverse branches $h_1,h_2\in\mathcal{H}_{n_0}$ for some sufficiently large integer $n_0\geqslant 1$, such that $\inf_{x\in J}|\psi'_{h_1,h_2}(x)|\geqslant D.$
}
\end{definition}
\smallskip

Notably, a roof function $r$ over a full branch piecewise expanding map satisfies the UNI condition if and only the 
roof function is not $C^1$-cohomologous to a piecewise constant one (cf. \cite[Proposition~7.4]{AGY}). 
The previous condition is often used to obtain a contraction of the $L^2$-norm of the twisted transfer operators.
To make it precise it is useful to write the
iterates of $\mathscr L_s$ as
\begin{eqnarray}\label{defAs}
(\mathscr L^n_sv)(x)=\lambda_{\sigma}^{-n}f^{-1}_{\sigma}(x)\, \sum_{h\in \mathcal{H}_n} A_{s,h,n}(f_{\sigma}v)(x)
	\;\text{and}\;
	A_{s,h,n}(f_{\sigma}v)(x)=e^{S_n (\phi-sr)(h(x))}(f_{\sigma}v)(h(x))
\end{eqnarray}
for every $x\in J$, where $S_n (\phi-sr)$
denotes the usual Birkhoff sum for the potential $\phi-s r$. 
The key estimate is given by the following cancellation argument, whose proof is contained in \cite[Lemma~2.9]{AM}, where the argument does not involve the potential (see also \cite{DV}).

\begin{lemma}\label{lemm:cancel} 
Let $\eta_0=\frac{1}{2}(\sqrt{7}-1)\in(\frac{2}{3},1)$. Assume that the UNI condition is satisfied with constants $D>0$ and $n_0\geqslant 1$ and inverse branches $h_1, h_2\in\mathcal{H}_{n_0}$.
There exists $\Delta=2\pi/D$ and $0<\delta<\Delta$ such that for all $s=\sigma+ib$, $|\sigma|<\varepsilon$, $|b|>4\pi/D$, and all $(u,v)\in\mathcal{C}_b$ the following holds:
for every $x_0\in I$ there exists $x_1\in B(x_0, \Delta/|b|)$ such that one of the following inequalities holds on $B(x_1,\delta/|b|)$:
\begin{enumerate}
\item[Case $h_1$:]
$|A_{s,h_1,n_0}(f_{\sigma}v)+A_{s,h_2,n_0}(f_{\sigma}v)|\leqslant \eta_0 A_{\sigma,h_1,n_0}(f_{\sigma}u)+A_{\sigma,h_2,n_0}(f_{\sigma}u)$,
\item[Case $h_2$:]
$|A_{s,h_1,n_0}(f_{\sigma}v)+A_{s,h_2,n_0}(f_{\sigma}v)|\leqslant A_{\sigma,h_1,n_0}(f_{\sigma}u)+ \eta_0A_{\sigma,h_2,n_0}(f_{\sigma}u)$.
\end{enumerate}
\end{lemma}

\begin{remark}\label{rem:UNI}
\emph{
It is known that the previous lemma ensures that there exist uniform constants $\Delta,\delta$ (depending on $D$)
so that the terms of the twisted transfer operator, observed as vectors in $\mathbb C$, exhibit cancellations
within a ball of radius $\delta/|b|$ inside any ball of larger radius $\Delta/|b|$. 
 This choice of constants, as the definition of the UNI condition itself, requires implicitly that the map $T$ is full branch. 
}
\end{remark}

There are two key issues in order to generalize the classical UNI condition to our context. First, as inverse branches
are not necessarily full branches the domain of the maps ~\eqref{eq:def-psi} 
may have boundary points.  
Second, as $T$ is not necessarily full branch, it may occur that the selection of inverse branches for the conclusion of 
Definition~\ref{def:UNI} may need to be pointwise. 
Hence one is tempted to define the following natural extension:

\begin{definition}\label{def:UNI-weak}
\emph{
The suspension semiflow $(X_t)_{t\geqslant 0}$ over a piecewise expanding Markov map $T$ with roof function $r$ satisfies the 
\emph{uniform non-integrability (UNI) condition}
if there exists $D,R>0$ and for some sufficiently large integer $n_0\geqslant 1$ and $y\in \bigcup_{i=1}^m P_i$ there exist
inverse branches $h_{\overline{\omega}},h_{\omega}\in\mathcal{H}_{n_0}$ such that:
\begin{enumerate}
\item $y\in T^{n_0}(\omega) \cap T^{n_0}(\overline{\omega})$, and 
\item $ \inf_{ x\in B(y,R) \cap \mbox{Dom}(\psi_{h_{\omega},h_{\overline{\omega}}})} |\psi'_{h_{\overline{\omega}},h_{\omega}}|\geqslant D.$
\end{enumerate}
}
\end{definition}
This condition appears naturally in our context  (cf. Proposition~\ref{prop:transv-implies-wUNI}). This turns out to be crucial in the construction of a 
cancellations argument identical to Lemma~\ref{lemm:cancel}, which now relies on the construction of suitable function determined (in open sets) according to the inverse branches that appear in Definition~\ref{def:UNI-weak}.

\subsection{Geometric aspects of transversality: cones and coboundaries}\label{geom-transv}
This subsection, inspired by \cite{BW,T0}, shows how transversality can be observed in terms of cone fields. The results in this subsection are independent on the dimension $d$ of the phase space $J$,  where $T$ acts as a conformal expanding map.
Given $t\geqslant 0$ one can write the suspension semiflow by 
$$
X_t(x, u)=(T^n(x), t+u-S_nr(x))
$$ 
where $n=n(t,(x,u))\geqslant 1$ is uniquely determined by $S_nr(x)\leqslant u+t<S_{n+1}r(x)$, a semiflow that evolves on
the $(d+1)$-dimensional quotient manifold $J^r$ obtained from $J\times \mathbb R_+$ as described in 
Subsection~\ref{subsec:susp}. After an identification of the tangent space at each point with 
$\mathbb R^{d+1}$ one can write $DX_t(x,u)=\mathcal{D}^{n}(x,u)$, where
$\mathcal{D}^n(x,u):\mathbb{R}^{d+1}\to\mathbb{R}^{d+1}$ is defined as
\begin{eqnarray*}
\mathcal{D}^{n}(x,u)=\begin{pmatrix}
DT^n(x) & 0 \\
-DS_nr(x)& 1
\end{pmatrix}
\end{eqnarray*}
(we shall omit the dependence of $\mathcal{D}^n(x,u)$ for notational simplicity).

Now, for each $x\in J$ consider the cone $\cC(x):=\cC\subset\mathbb{R}^{d+1}$ defined by
\begin{eqnarray*}
\cC:=\left\{{a\choose b}
:a\in\mathbb{R}^d,\,b\in\mathbb{R},\, |b|\leqslant C_7\|a\|\right\},
\end{eqnarray*}
where $C_7>0$ 
was defined by Lemma~\ref{lemma1} above.
We claim that this cone-field is strictly invariant under the action of 
$\mathcal{D}^n(x)$, for every $x\in J$.
Indeed, if $x\in J$ and ${a\choose b}\in\cC$ then 
\begin{eqnarray*}
\mathcal{D}(x){a\choose b}=\begin{pmatrix}
DT(x) & 0 \\
-Dr(x)& 1
\end{pmatrix}\begin{pmatrix}
a\\
b 
\end{pmatrix}=\begin{pmatrix}
a_1\\
b_1 
\end{pmatrix},
\end{eqnarray*}
with $a_1=DT(x)a$ and $b_1=b-Dr(x)a$.
Choosing $\omega\in\cP$ so that $a=Dh_{\omega}(Tx)a_1$, one can use the backward contraction of $T$ (recall \eqref{h-1}) and hypothesis \eqref{h1} to get
\begin{eqnarray*}
|b_1|=|b-Dr(x)a|=|b-D(r\circ h_{\omega})(Tx)a_1|&\leqslant &|b|+|D(r\circ h_{\omega})(Tx)a_1|\\
&\leqslant & 
 C_7\lambda^{-1} \, \|a_1\|+C_4\|a_1\|\leqslant C_7\|a_1\|.  
\end{eqnarray*}
Since $x$ was chosen arbitrary and $\mathcal{D}^{n}(x)=\mathcal{D}(T^{n-1}(x)) \cdot \dots \cdot \mathcal{D}(x)$, a recursive argument ensures that all matrices $\mathcal{D}^{n}(x)$ preserve the cone $\mathcal C$, as claimed.

\smallskip
Now, following \cite{BW}, we recall the concept of transversality involving the images of the cone $\mathcal C$ by the matrices $\mathcal D^n(x)$ 
associated to different pre-images of a point. 

\begin{definition}\label{def:transversality}
\emph{
Given $y\in J$ and a pair of points $x_1, x_2\in T^{-n}(y)$, 
we say that the corresponding image cones are \emph{transversal} (and denote it by $\mathcal{D}^n(x_1)\cC\pitchfork\mathcal{D}^n(x_2)\cC$)
if the intersection
$
\mathcal{D}^n(x_1)\cC\cap\mathcal{D}^n(x_2)\cC
$ 
does not contain a $d$-dimensional subspace.}
\end{definition}

\begin{remark}\label{consq_transv}
\emph{
We will use the following consequences of transversality which can be found in \cite[Lemma~3.7]{BW}:
\begin{enumerate}
\item
If $\omega,\overline{\omega}\in\mathcal{P}^{(n)}, y\in J$ and 
$\mathcal{D}^n(h_{\omega}y)\cC\pitchfork\mathcal{D}^n(h_{\overline{\omega}}y)\cC$
then there exists a one-dimensional subspace $E\subset\mathbb{R}^d$ such that
\begin{eqnarray*}
|D(S_nr\circ h_{\omega})(y)v-D(S_nr\circ h_{\overline{\omega}})(y)v|>C_7(\|Dh_{\omega}(y)v\|+\|Dh_{\overline{\omega}}(y)v\|),
\;  \forall v\in E.
\end{eqnarray*}
The previous consequence of the transversality condition can be understood as a 
UNI condition along some $d$-dimensional subspace. Observe that
any lower bound for the left hand-side above needs to be exponentially decreasing to zero as $n$ tends to infinity (see \eqref{eq:inversebranch} below).
\item
For each $n\in\mathbb{N}$ the real number
\begin{eqnarray}\label{eq:def-an}
a(n):=\sup_{y\in J}\sup_{x_0\in T^{-n}(y)} \,\lambda_0^{-n} \,\frac{1}{f_0(y)} \,\sum_{\substack{x\in T^{-n}(y)\\ \mathcal{D}^n(x)\cC\,\not\pitchfork\,\mathcal{D}^n(x_0)\cC}}e^{S_n\phi(x)}f_0(x)
\end{eqnarray}
satisfies $a(n)\leqslant  \|{\mathscr L}_0^n \mathbf 1\|_\infty =1 $. In particular $\displaystyle\limsup_{n\to\infty}{a(n)}^{\frac{1}{n}}\leqslant 1$.
\end{enumerate}
}
\end{remark}

We will see below that transversality can be characterized in terms of the roof function not being cohomologous to a piecewise constant roof function or, alternatively, in terms of the growth rate of the sequence $(a_n)_{n\in\mathbb{N}}$. 
In order to prove such a characterization we define the sequence $(b(n))_{n\in\mathbb{N}}$ in the following way. For each $n\in\mathbb{N}, y\in J$ and $d$-dimensional subspace $P\subset\mathbb{R}^{d+1}$ consider the quantities
\begin{eqnarray}\label{def:bn}
b(n,P,y):=\sum_{\substack{x\in T^{-n}(y)\\ \mathcal{D}^n(x)\cC\supset P}}\frac{1}{J_{\mu_{0}}T^n(x)}
	\quad\text{and} \quad
	b(n):=\sup_{y}\sup_{P} b(n,P,y),
\end{eqnarray}
where  
$J_{\mu_{0}}T^n(x)=\lambda_0^{-n} e^{S_n\phi(x)} \frac{f_0(x)}{f_0(T^n(x))}$ 
is the Jacobian of the invariant probability $\mu_{0}$ with respect to $T^n$. 
Observe that $b(n,P,y) \leqslant \sum_{x\in T^{-n}(y)} (J_{\mu_{0}}T^n(x))^{-1}=1$ for every $n\geqslant 1$ and $y\in J$.
Moreover, it is not hard to check that 
$(b(n))_{n\in \mathbb N}$ is submultiplicative.
The following is similar to \cite[Lemma~3.8]{BW} for 
transfer operators associated to more general potentials.

\begin{lemma}\label{transv2} 
Let $T:J\to J$ be a $C^{1+\alpha}$ piecewise expanding Markov map and let $r:J\to\mathbb{R}^{+}$ satisfy the previous hypotheses. The following properties are equivalent:
\begin{itemize}
\item[(i)] $\displaystyle\lim_{n\to\infty}a(n)^{\frac{1}{n}}=1$;
\item[(ii)]$\displaystyle\lim_{n\to\infty}b(n)^{\frac{1}{n}}=1$;
\item[(iii)] For all $n\in\mathbb{N}$ and $y\in J$ there exists a $d$-dimensional subspace $E^{(n)}_y\subset\cC$ such that $\mathcal{D}^n(x)\cC\supset E^{(n)}_y$ for all $x\in T^{-n}(y)$;
\item[(iv)] There exists a piecewise $C^1$-smooth map 
$\theta: J \to\mathbb R$ such that $r=\theta\circ T-\theta+\chi$ where $\chi$ is constant on each partition element. 
\end{itemize}
\end{lemma}

\begin{proof}
Let us prove the several implications separately. 

\medskip
\noindent $(i)\Longrightarrow (ii)$

Suppose that $\displaystyle\lim_{n\to\infty}a(n)^{\frac{1}{n}}=1$. In order to prove (ii) it is enough to show that there exists a sequence 
$(n_k)_{k\in\mathbb{N}}$ of positive integers so that $a(n_k)\leqslant b(k)\leqslant 1$ for all $k\in\mathbb{N}$. 
For each $n\geqslant 1$ and $x \in T^{-n}(y)$ observe that the 
$d$-dimensional subspace $Q_n(x_1):= \mathcal{D}^n(x_1)(\mathbb{R}^d\times \{0\})$  belongs to the cone 
$\mathcal{D}^{n}(x_1)\cC$.  We need the following geometric consequence of domination:

\smallskip 
\noindent {\bf Claim:}\emph{
Given $n\geqslant 1$ there exists $1\leqslant \ell \leqslant n$ so that if $n=k+\ell$, $y\in J$, $x_1,x_2 \in T^{-n}(y)$, 
and
$\mathcal{D}^n(x_1)\cC\,\not\pitchfork\,\mathcal{D}^n(x_2)\cC$ then $\mathcal{D}^{k}(T^{\ell}x_2)\cC\supset Q_n(x_1)$. 
 }
 \begin{proof}[Proof of the claim]
Given $n\geqslant 1$ and $1\leqslant k\leqslant n$, we start by noticing that
\begin{eqnarray}\label{eq:images}
\mathcal{D}^k(x)\cC:=\left\{{a\choose b-DS_kr(x)DT^{-k}(T^kx)a}
:a\in\mathbb{R}^d,\,b\in\mathbb{R},\, |b|\leqslant C_7\|DT^{-k}(T^kx)a\|\right\}
\end{eqnarray}
 (for notational simplicity we let $DT^{-k}(T^kx)$ denote
the derivative at $T^k(x)$ of the inverse branch of $T^k$ containing that point).
Fix $x_1,x_2\in T^{-n}(y)$ so that $\mathcal{D}^n(x_1)\cC\,\not\pitchfork\,\mathcal{D}^n(x_2)\cC$. Assume that $x_1$ and $x_2$ are distinct, otherwise there is nothing to prove.
The argument in the proof consists in getting a cone $\widehat {\mathcal C}$ such that $\mathcal{D}^{k}(T^{\ell}x_2)\cC\supset\widehat {\mathcal C}\supset\mathcal{D}^{n}(x_2) \mathcal C$ that can be used to quantify the fact that $Q_n(x_1)$ and the cone
$\mathcal{D}^{n}(x_2)\cC$ remain within a distance $C_7\lambda ^{-n}$ in a projective metric, then allowing to choose a suitable $1\leqslant \ell \leqslant n$ (depending only on hyperbolicity rates) so that $Q_n(x_1)\subset\mathcal{D}^{k}(T^{\ell}x_2)$. Let us be more precise. Consider the enlargement 
\begin{eqnarray*}
\widehat {\mathcal C}
=
\left\{{a\choose b_1+b_2}
:\,{a\choose b_1}\in\mathcal{D}^{n}(x_2)\cC,\, |b_2|\leqslant C_7\lambda ^{-n}\|a\|\right\},
\end{eqnarray*}
of the cone $\mathcal{D}^{n}(x_2)\mathcal C$, which is also a cone. Using ~\eqref{eq:images}, in order to select $1\leqslant \ell \leqslant n$ so that $\mathcal{D}^{k}(T^{\ell}x_2)\cC\supset\widehat {\mathcal C}$ we need to ensure that each vector in $\widehat {\mathcal C}$ 
can be written as
\begin{eqnarray*}
{a\choose b-DS_{k}r(T^{\ell}x_2)DT^{-k}(T^nx_2)a},\,\,\,\text{for some}\,\,\, |b|\leqslant C_7\|DT^{-k}(T^n)a\|.
 \end{eqnarray*}
Indeed, given ${a\choose b_1+b_2}\in \widehat {\mathcal C}$ define $b=b_1+b_2+DS_{k}r(T^{l}x_2)DT^{-k}(T^nx_2)a$. 
As  ${a\choose b_1}\in \mathcal{D}^{n}(x_2)\cC$, using once more ~\eqref{eq:images}, one can write
$b_1=b_0-DS_{n}r(x_2)DT^{-n}(T^nx_2)a$ with $|b_0|\leqslant C_7\|DT^{-n}(T^nx_2)a\|$. Thus, 
using Lemma~\ref{lemma1} and the chain rule in $S_{k+\ell}r=S_{k}r\circ T^{\ell}+S_{\ell}r$ one deduces that 
\begin{eqnarray*}
|b|&=&|b_1+b_2+DS_{k}r(T^{\ell}x_2)DT^{-k}(T^nx_2)a|\\
&\leqslant & |b_0+b_2-DS_{\ell}r(x_2)DT^{-\ell}(T^{\ell}x_2)DT^{-k}(T^nx_2)a|\\
&\leqslant & C_7(2\lambda ^{-n}\|a\|+\dfrac{1}{2}\|DT^{-k}(T^nx_2)a\|)\\
&\leqslant & C_7(2\lambda ^{-n}\|a\|-\dfrac{1}{2}\|DT^{-k}(T^nx_2)a\|)+{C_7}\|DT^{-k}(T^nx_2)a\|.
\end{eqnarray*}

We now claim that the first summand in the last expression is negative. Indeed,  
since 
$\lambda \leqslant  \|DT(x)^{-1}\|^{-1}\leqslant  \|DT(x)\|\leqslant {\rho}$
 for all $x\in J$ we get that $\|DT^{-k}(T^nx_2)a\|\geqslant \rho^{-k}\|a\|$ and, consequently, $2\lambda ^{-n}\|a\|<\frac{1}{2}\|DT^{-k}(T^nx_2)a\|$ whenever $n\geqslant k\frac{\log(\rho/4)}{\log\lambda}$. Thus it is enough to take
$n=n_k:=\lfloor k\frac{\log(\rho/4)}{\log\lambda}\rfloor$. 
\end{proof}

We can now complete the proof of this implication. Indeed, taking $n=n_k=\lfloor k\frac{\log(\rho/4)}{\log\lambda}\rfloor$ given the previous claim, 
writting $n=k+\ell$ and using that ${\mathscr L}^\ell_0 \mathbf 1=\mathbf 1$, we conclude that 
\begin{eqnarray*}
\,\lambda_0^{-n} \,  \frac{1}{f_0(y)}\sum_{\substack{x_2\in T^{-n}(y)\\ \mathcal{D}^n(x_2)\cC\,\not\pitchfork\,\mathcal{D}^n(x_1)\cC}}e^{S_n\phi(x_2)}f_0(x_2)
&\leqslant & \,\lambda_0^{-n} \,  \frac{1}{f_0(y)}\displaystyle\sum_{\substack{x_2\in T^{-n}(y)\\ \mathcal{D}^k(T^\ell x_2)\cC\supset Q_n(x_1)}}e^{S_k\phi(T^\ell x_2)}e^{S_\ell \phi(x_2)}f_0(x_2)\\
&\leqslant & \,\lambda_0^{-k} \,  \frac{1}{f_0(y)}\displaystyle\sum_{\substack{x_3\in T^{-k}(y)\\ \mathcal{D}^k(x_3)\cC\supset Q_n(x_1)}}e^{S_k\phi(x_3)} \cdot \lambda_0^{-\ell}\displaystyle\sum_{x_2\in T^{-\ell}(x_3)}e^{S_\ell\phi(x_2)}f_0(x_2)\\
&\leqslant & b(k,Q_{n}(x_1),y) \cdot ({\mathscr L}^\ell_0 \mathbf 1)(x_3) \smallskip \\ 
&= & b(k,Q_{n}(x_1),y)
\end{eqnarray*}
for any $x_1\in T^{-n}(y)$. This proves that $a(n_k)\leqslant b(k)\leqslant 1$ for every $k\geqslant 1$, as claimed. 

\smallskip
\noindent $(ii)\Longrightarrow (iii)$
\smallskip

By construction the sequence $(b(n))_{n\in\mathbb{N}}$ is submultiplicative, bounded above by $1$ and satisfies $\displaystyle\lim_{n\to\infty}b(n)^{\frac{1}{n}}=1$, 
hence $b(n)=1$ for all $n\geqslant 1$.
Therefore, there exists $y_n\in J$
and a $d$-dimensional subspace $E^{(n)}\subset\mathbb{R}^{d+1}$ such that
$$
b(n,E^{(n)},y_n)=\sum_{\substack{x\in T^{-n}(y_n)\\ \mathcal{D}^n(x)\cC\supset E^{(n)}}}\frac{1}{J_{\mu_{0}}T^n(x)}=1
$$
and, in particular, $\mathcal{D}^{n}(x)\cC\supset E^{(n)}$ for every $x\in T^{-n}(y_n)$.
We claim that this property implies property (iii) holds. Indeed, if this was not the case 
there would exist $n_0\in\mathbb{N}, y_0\in J$ and $x_1, x_2\in T^{-{n_0}}(y_0)$ such that $\mathcal{D}^{n_0}(x_1)\cC\cap \mathcal{D}^{n_0}(x_2)\cC$ does not contain a $d$-dimensional subspace.
Choose $\omega_1, \omega_2\in\mathcal{P}^{(n_0)}$ and $h_{\omega_1}, h_{\omega_2}\in \mathcal H_{n_0}$ so that $x_1=h_{\omega_1}(y_0)$ and $ x_2=h_{\omega_2}(y_0)$. By the openness of the transversality property there exists an open neighborhood $J_{12} \subset \mbox{Dom}(h_{\omega_1}) \cap \mbox{Dom}(h_{\omega_2})$ of $y_0$ such that $\mathcal{D}^{n_0}(h_{\omega_1}(y))\cC\cap \mathcal{D}^{n_0}(h_{\omega_2}(y))\cC$ does not contain a $d$-dimensional subspace for all $y\in J_{12}$.
Now, we make use of the covering property of $T$: there exists $m_0\geqslant 1$ so that $T^{m_0}(J_{12})=J$.
In consequence, for every $x\in J$ there exists $\omega_0\in\mathcal{P}^{(m_0)}$ such that $h_{\omega_0}(x) \in J_{12}$. 
Set $x_1=h_{\omega_1}(h_{\omega_0}x)$ and $x_2=h_{\omega_2}(h_{\omega_0}x)$. Using that
\begin{eqnarray*}
\mathcal{D}^{n_0+m_0}(x_1)\cC\cap\mathcal{D}^{n_0+m_0}(x_2)\cC
	=\mathcal{D}^{m_0}(h_{\omega_0}x)\, \Big(\mathcal{D}^{n_0}(h_{\omega_1}(y))\cC\cap\mathcal{D}^{n_0}(h_{\omega_2}(y))\cC\Big),
\end{eqnarray*}
we conclude that all points $x\in J$ have some preimages $x_1, x_2\in T^{-(m_0+n_0)}(x)$ 
such that $\mathcal{D}^{n_0+m_0}(x_1)\cC\cap\mathcal{D}^{n_0+m_0}(x_2)\cC$ does not contain a $d$-dimensional subspace of $\mathbb{R}^{d+1}$, thus leading to a contradiction with the existence of the points $y_n$. This proves that item (iii) 
holds. 

\smallskip
\noindent $(iii)\Longrightarrow (iv)$
\smallskip

Given an integer $n\geqslant 1$ and a collection $\omega_1,\omega_2, \dots, \omega_n$ of elements in the Markov partition $\cP$, set $h_n:=h_{\omega_n}\circ\cdots\circ h_{\omega_2}\circ h_{\omega_1}$. The chain rule guarantees that
\begin{eqnarray*}
D(S_nr\circ h_n)(x)=\sum^{n}_{k=1}D(r\circ h_{\omega_k})(h_{k-1}(x))Dh_{k-1}(x),
	\qquad\forall x\in \mbox{Dom}(h_n).
\end{eqnarray*}
The argument in the proof of Lemma \ref{lemma1} implies that the previous series is absolutely convergent. We claim that property (iii) implies that this convergence does not depend of the choice of sequence of inverse branches. Indeed, saying that for all $n$ and $y\in J$ there exists a $d$-dimensional subspace $E^{(n)}_y\subset\cC$ such that $\mathcal{D}^n(x)\cC\supset E^{(n)}_y$ for all $x\in T^{-n}(y)$ implies that $\mathcal{D}^n(x_1)\cC\,\not\pitchfork\,\mathcal{D}^n(x_2)\cC$ for each $x_1, x_2\in T^{-n}(y)$.
Take a vector $\overline{a}\in E^{(n)}_y\subset\mathcal{D}^n(x_1)\cC\cap\mathcal{D}^n(x_2)\cC$. 
Recalling ~\eqref{eq:images}, there are $a\in\mathbb{R}^d, b_1,b_2\in\mathbb{R}$ such that  
\begin{eqnarray*}
\overline{a}={a\choose b_1-DS_nr(x_1)DT^{-n}(T^nx_1)a}={a\choose b_2-DS_nr(x_2)DT^{-n}(T^nx_2)a},
\end{eqnarray*}
with $|b_1|\leqslant C_5\|T^{-n}(T^nx_1)a\|$ and $|b_2|\leqslant C_5\|T^{-n}(T^nx_2)a\|$. 
Hence 
\begin{eqnarray}\label{eq:inversebranch}
|DS_nr(x_1)DT^{-n}(T^nx_1)a-DS_nr(x_2)DT^{-n}(T^nx_2)a| =  |b_1-b_2|\leqslant 2C_7\lambda ^{-n}\|a\|
\end{eqnarray}
or, equivalently,
$
|D(S_nr\circ h_{\omega})(y)a-D(S_nr\circ h_{\overline{\omega}})(y)a|\leqslant 2C_7\lambda ^{-n}\|a\|,
$
where $\omega,\overline{\omega}\in \cP^{(n)}$ satisfy $h_{\omega}(y)=x_1$ and $h_{\overline{\omega}}(y)=x_2$. Now, define 
$$
l(y):=\displaystyle\lim_{n\to\infty}D(S_nr\circ h_n)(y)
	= \displaystyle\lim_{n\to\infty} \sum_{j=0}^{n-1} D(r\circ h_{n-j})(y).
$$ 
whose limit is independent of the choice of inverse branches which define $h_n$, by \eqref{eq:inversebranch}.
A simple computation shows that
\begin{eqnarray}\label{cohomol}
l(y)=D(r\circ h_{\omega})(y)+l(h_{\omega}y) Dh_{\omega}(y), \qquad \forall \omega\in\cP 
	\, \, \text{s.t. $y\in \mbox{Dom}(h_\omega)$}.
\end{eqnarray}
Now, for each $\omega_1\in\cP$ choose $y_0\in \mbox{Dom}(h_{\omega_1})$ and define
\begin{eqnarray*}
\theta(y)=\sum^{\infty}_{j=1} [(r\circ h_j)(y)-(r\circ h_j)(y_0)], \qquad y\in \mbox{Dom}(h_{\omega_1}).
\end{eqnarray*}
The previous discussion ensures that the expression does not depend on the point $y_0 \in \mbox{Dom}(h_{\omega_1})$. Exhausting the elements of the partition $\mathcal P$ we obtain 
a piecewise ${C}^1$-smooth function $\theta: J \to \mathbb R$
which satisfies $D\theta(y)=l(y)$. Furthermore, by \eqref{cohomol} one concludes that $D(r+\theta-\theta\circ T)=0$ at all points in the interior of the elements of the Markov partition, which ensures that $r+\theta-\theta\circ T$ is piecewise constant.

\smallskip
\noindent $(iv)\Longrightarrow (i)$
\smallskip

Assume there exists a piecewise $C^1$-smooth function $\theta: J \to \mathbb{R}$ 
such that $r=\theta\circ T-\theta+\chi$ where $\chi$ is constant on each partition element. 
Then, taking $h_n\in \mathcal H_n$, differentiating the expression $S_n r \circ h_n =\theta -\theta\circ h_n + S_n\chi\circ h_n$
and using Lemma~\ref{lemma1} we conclude that $\|D\theta(x)\| \leqslant C_7$ for each $x\in J$, and so
the subspace
\begin{eqnarray*}
\mathcal{E}(x):=\left\{{a\choose {-D\theta (x)a}}\,:\,a\in\mathbb{R}^d\right\}\subset\mathbb{R}^{d+1}
\end{eqnarray*}
is contained in the cone $\cC$. 
 Using that $DS_nr(x)=D\theta(T^nx)DT^n(x)-D\theta(x)$ we also get that
\begin{eqnarray*}
\mathcal{D}^{n}(x)\cE(x) &=&\left\{\begin{pmatrix}
DT^n(x) & 0 \\
-DS_nr(x)& 1
\end{pmatrix}{a\choose {-D\theta (x)a}}\,:\,a\in\mathbb{R}^d\right\}\\
&=&\left\{{{DT^{n}(x)a}\choose {-D\theta(T^{n}x)DT^{n}(x)a}}\,:\,a\in\mathbb{R}^d\right\}=\mathcal{E}(T^nx),
\end{eqnarray*}
which ensures that 
$\mathcal C \supset \mathcal{D}^n(x)\cC\supset\mathcal{D}^n(x)\cE(x)=\cE(y)$ for all $y\in J$ and $x\in T^{-n}(y)$.
We conclude that $\mathcal{D}^n(x)\cC\,\not\pitchfork\,\mathcal{D}^n(x_0)\cC$ for all $x,x_0\in T^{-n}(y)$, which implies that actually $a(n)=({\mathscr L}_0^n\mathbf 1)(y)=1$ for all $y\in J$ and $n\geqslant 1$. 
This proves that property (i) holds and finishes the proof of the lemma.
\end{proof}

\begin{remark}\label{transvxUNI}
\emph{
In the case that $e^{-\gamma_0}:=\liminf_{n\to \infty}a(n)^{\frac{1}{n}}<1$ we get that 
for any $\gamma\in (0,\gamma_0)$ there exists $C_{\gamma}>0$ such that $a(n)\leqslant C_{\gamma}e^{-\gamma n}$ for  all $n\in\mathbb{N}$. Hence for all $y\in J$ and $x_0\in T^{-n}(y)$,
\begin{eqnarray}
\frac{1}{f_0(y)}\sum_{\substack{x\in T^{-n}(y)\\ \mathcal{D}^n(x)\cC\,\not\pitchfork\,\mathcal{D}^n(x_0)\cC}}e^{S_n\phi(x)}f_0(x)\leqslant C_{\gamma}e^{-\gamma n}, \quad \forall n\geqslant 1.
\end{eqnarray}
As the right hand side above tends to zero 
we conclude that for each $y\in J$ and every large $n\geqslant 1$ there exists a pair of preimages $x_0,x_1\in T^{-n}(y)$ satisfying $\mathcal{D}^{n}(x_0)\cC\pitchfork\mathcal{D}^{n}(x_1)\cC$.
}
\end{remark}
\subsection{Analytic aspects of transversality: a cancellations lemma}\label{ssubsec:weak}

Here we prove that transversality gives rise to a cancellations lemma similar to Lemma~\ref{lemm:cancel}.
We proceed to show that transversality implies the UNI condition for the suspension semiflow associated to the piecewise expanding Markov map on the base. 
The cancellations argument following them require the map $T$ to act on the interval. We begin with the following auxiliary lemma.

\begin{lemma}\label{le:UNI}
There exists $C_{9}>0$ such that, for all $n\in\mathbb{N}, \omega, \overline{\omega}\in\mathcal{P}^{(n)},$
\begin{eqnarray*}
|D\psi_{\omega,\overline{\omega}}(x)v-D\psi_{\omega,\overline{\omega}}(y)v| \leqslant C_{9}\, \|v\| \; d(x,y)^\alpha, \; 
\quad \forall x,y\in \mbox{Dom}(\psi_{\omega,\overline{\omega}}), \; \forall v\in \mathbb R^d.
\end{eqnarray*}
\end{lemma}
\begin{proof}
Write 
$
D(S_nr\circ h_{\omega})(x)=\sum^{n-1}_{k=0} Dr(T^kh_{\omega}x)D(T^k\circ h_{\omega})(x)
$
for each $x \in \mbox{Dom}(h_\om)$.
Given $x,y \in \mbox{Dom}(\psi_{\omega,\overline{\omega}})$, using the triangular inequality 
and the uniform H\"older assumption on the derivative of inverse branches (recall ~\eqref{h0}), we obtain 
\begin{align*}
 |D(S_nr\circ h_{\omega})(x)v - & D(S_nr\circ h_{\omega})(y)v| \\
	& \leqslant \sum^{n-1}_{k=0} |Dr(T^kh_{\omega}x)D(T^k\circ h_{\omega})(x)v 
				- Dr(T^kh_{\omega}x)D(T^k\circ h_{\omega})(y) v| \\
	& + \sum^{n-1}_{k=0} |Dr(T^kh_{\omega}x)D(T^k\circ h_{\omega})(y)v - Dr(T^kh_{\omega}y)D(T^k\circ h_{\omega})(y)v| 
	\\
	& \leqslant  \|Dr\|_{\infty} \|v\| \sum^{n-1}_{k=0} \| D(T^k\circ h_{\omega})(x) - D(T^k\circ h_{\omega})(y)\| \\
	& + |Dr|_{{C}^{\alpha}(J)} \|v\| \sum^{n-1}_{k=0} \lambda^{-n+k} \; d(T^kh_{\omega}x, T^kh_{\omega}y )^\alpha \\
	&\leqslant \max\{C_2,1\} \, \|Dr\|_{{C}^{\alpha}(J)} \;\|v\| \, \sum^{n-1}_{k=0}\lambda^{-(n-k)} \; d(x,y)^{\alpha},
\end{align*}
where $|Dr|_{{C}^{\alpha}(J)}$ (resp. $\|Dr\|_{{C}^{\alpha}(J)}$) stands for the maximum of the H\"older constant 
(resp. H\"older norm) of $Dr$ among the Markov domains. 
As the same estimates hold for the cylinder $\bar\om\in\mathcal P^{(n)}$, the lemma holds with 
$C_{9}:=2 \max\{C_2,1\}  \|Dr\|_{{C}^{\alpha}(J)} \, (1-\lambda^{-1})^{-1}$.
\end{proof}

The following key proposition, which bridges between two very natural geometric
properties of quite different nature shows that the transversality condition can actually ensure that a 
UNI condition along some $d$-dimensional subspace.  
Let us proceed as follows. 

\medskip
Assuming that the transversality condition holds, there exists $N\geqslant 1$ so that
\begin{eqnarray*}\label{eq:N}
\,\lambda_0^{-n} \,  \frac{1}{f_0(y)}\sum_{\substack{x\in T^{-n}(y)\\ \mathcal{D}^n(x)\cC\,\not\pitchfork\,\mathcal{D}^n(x_0)\cC}}e^{S_n\phi(x)}f_0(x)\leqslant C_{\gamma}e^{-\gamma n} < 1
\end{eqnarray*}
for all $y\in J$, $x_0\in T^{-n}(y)$ and $n\geqslant N$ (recall  Remark~\ref{transvxUNI}).
Throughout, let $0<\delta<1$ be small and fixed so that 
\begin{equation}\label{eq:def-consts0}
n_2:=\left\lfloor\frac{\log\delta}{-\log\rho}\, \right\rfloor  \geqslant N.
\end{equation}
Take also the integer
\begin{equation}\label{eq:def-consts}
n:=n_1+n_2, 
\quad
\text{where $n_1\geqslant 1$ will be chosen in ~\eqref{defxi}.}
\end{equation}
By the choice of $n_2$, for each point $y\in J$ there exist at least two transversal preimages in $T^{-n_2}(y)$ (recall Remark~\ref{transvxUNI}). In particular, the assumption of the following proposition is always satisfied for some pair of cylinder sets. 

\begin{proposition}\label{prop:transv-implies-wUNI}
There are constants $b_0>1$, $\Delta>1$ and $D>0$ (depending only on $\delta$) 
so that the following holds: if $y\in J$ and $\omega,\overline{\omega}\in\mathcal{P}^{(n)}$ satisfy
$
\mathcal{D}^{n_2}(T^{n_1}h_{\overline{\omega}}y)\cC\pitchfork\mathcal{D}^{n_2}(T^{n_1}h_{\omega}y)\cC
$
then there exists a unit vector $v\in \mathbb R^d$ so that
\begin{eqnarray*}
| D\psi_{\omega,\overline{\omega}}(z)v| 
	=\left| D(S_nr\circ h_{\overline{\omega}})(z) v - D(S_nr\circ h_{\omega})(z)v \right|
	\geqslant D
\end{eqnarray*}
for every $z \in B\big(y, \frac{\Delta}{|b|}\big) \cap \mbox{Dom}(\psi_{\omega,\overline{\omega}})$ and every $|b|>b_0$.
\end{proposition}

\begin{proof}
Fix $y\in J$ and $\omega,\overline{\omega}\in\mathcal{P}^{(n)}$ as in the hypothesis.
By item (1) in Remark~\ref{consq_transv} (taking $n=n_2$ and the inverse branches $T^{n_1}\circ h_{\omega}$ and
$T^{n_1}\circ h_{\bar \omega}$ for $T^{n_2}$)
ensures that there exists a one-dimensional subspace $E\subset\mathbb{R}^d$ and a unit vector $v\in E$ such that 
\begin{eqnarray*}
|D(S_{n_2}r\circ T^{n_1}\circ h_{\omega})(y)v-D(S_{n_2}r\circ T^{n_1}\circ h_{\overline{\omega}})(y)v|
	>C_7 (\|D(T^{n_1}\circ h_{\omega})(y)v\|+\|D(T^{n_1}\circ h_{\overline{\omega}})(y)v\|).
\end{eqnarray*}
On the other hand, Lemma \ref{lemma1} ensures that 
\begin{align*}
\max\Big\{ |D(S_{n_1}r\circ h_{\omega}) &(y)v|,  \; |D(S_{n_1}r\circ h_{\overline{\omega}})(y)v| \Big\} 
	\\
	& = \max\Big\{ |D(S_{n_1}r\circ h^1_{\omega})(h^2_\omega(y)) Dh^2_\omega(y)v|,  \; 
		|D(S_{n_1}r\circ h^1_{\bar \omega})(h^2_{\bar\omega}(y)) Dh^2_{\bar \omega}(y)v| \Big\} 
	\\
	& \leqslant \frac{C_7}{2} \min \Big\{ \| D(T^{n_1}\circ h_{\omega})(y)v\|,\; D(T^{n_1}\circ h_{\overline \omega})(y)v\|\Big\}.
\end{align*}
where $h_\omega=h^1_\omega \circ h^2_\omega$, \, $h_{\bar \omega}=h^1_{\bar \omega} \circ h^2_{\bar \omega}$, \,
$h^1_{\omega},h^1_{\bar \omega}$ are inverse branches of $T^{n_1}$ and $h^2_{\omega}, h^2_{\bar \omega}$ 
are inverse branches of $T^{n_2}$ (here we used that $h^2_{\bar \omega}=T^{n_1}\circ h^1_{\bar \omega}$ and 
$h^2_{ \omega}=T^{n_1}\circ h^1_{ \omega}$). 
Now, as $S_nr=S_{n_1}r+S_{n_2}r\circ T^{n_1}$ and $v$ is a unit vector, combining 
the previous estimates we get that
\begin{align*}
|D(S_{n}r\circ h_{\omega})(y)v-D(S_{n}r\circ h_{\overline{\omega}})(y)v|
	& \geqslant |D(S_{n_2}r\circ T^{n_1}\circ h_{\omega})(y)v-D(S_{n_2}r\circ T^{n_1}\circ h_{\overline{\omega}})(y)v|
	\\
	& \; - |D(S_{n_1}r\circ h_{\omega})(y)v|  - |D(S_{n_1}r\circ h_{\overline{\omega}})(y)v|
	\\
	& > \frac{C_7}{2}(\|D(T^{n_1}\circ h_{\omega})(y)v\|+\|D(T^{n_1}\circ h_{\overline{\omega}})(y)v\|)
	\\
	& >  C_7\rho^{-n_2}. 
\end{align*}
In other words, the previous expression shows that 
$
| D\psi_{\omega,\overline{\omega}}(y)v| 
	\geqslant  C_7\rho^{-n_2}. 
$
On the other hand, as $\|v\|=1$ and $D\psi_{\omega,\overline{\omega}}$ is a H\"older continuous operator 
with H\"older constant bounded above by $C_9$
(recall Lemma~\ref{le:UNI}), we obtain that 
$$
|D\psi_{\omega,\overline{\omega}}(z)v-D\psi_{\omega,\overline{\omega}}(y)v|
	\leqslant C_9 \,d(z,y)^{\alpha}
	\leqslant   \frac{C_7}{2} \rho^{-n_2} 
	\quad \text{for all} \; z\in B\big(y, \frac{\Delta}{|b|}\big) 
					 \cap \mbox{Dom}(\psi_{\omega,\overline{\omega}}),
$$ 
taking $\Delta=\frac{4\pi}{C_7\delta }$ (which depends on $\delta$). In the second inequality above we require that $|b|\geqslant b_0$ where $b_0>0$ is 
large so that 
$\frac{\Delta}{b_0}\leqslant  ( \frac{C_7  \rho^{-n_2}  }{2C_9})^{\frac1\alpha}$.
Altogether, this shows that 
$$
| D\psi_{\omega,\overline{\omega}}(z)v| 
	\geqslant   \frac{C_7}{2} \rho^{-n_2} 
		\quad \text{for every} \; z\in B\big(y, \frac\Delta{|b|}\big) 
					 \cap \mbox{Dom}(\psi_{\omega,\overline{\omega}})
$$
and completes the proof of the proposition, taking $D:=  \frac{C_7}{2} \rho^{-n_2}$. 
\end{proof}

Let us make some comments. First, Lemma~\ref{transv2} ensures that the absence of
the transversality condition implies on very rigid properties. Second, Proposition~\ref{prop:transv-implies-wUNI} 
says that transversality implies on the UNI condition evenly among points that are simultaneously 
on balls of radius $\frac\Delta{|b|}$ and contained in the domain of the inverse branches.  

\begin{remark}\label{rmk:choice-constants}
\emph{
Even though the conclusion of Proposition~\ref{prop:transv-implies-wUNI} is identical to the UNI condition we emphasize that the order of the constants is chosen in a substantially different way from the classical context, where the UNI condition 
provides an $n$ at which derivatives are bounded away from zero by a uniform constant, $\Delta>0$ can be chosen arbitrary small and $0<\delta<\Delta$ (depending on $\delta$) is taken such that cancellations occur inside a ball of radius $\delta/|b|$
for all large $|b|$. 
}
\end{remark}

At this point we shall use the uniform nonintegrability condition to obtain oscillatory cancellations for twisted transfer operator for certain classes of functions.
Recall that we denoted by $A_{s,h_{\omega},n}$ the terms appearing in the iterations of the twisted transfer operator (recall ~\eqref{defAs}).
For each $b\in\mathbb{R}$, consider the \textit{cone} $\mathcal{C}_b$ defined by
\begin{eqnarray}\label{def:coneCb}
\mathcal{C}_b=\Big\{\,(u,v)\hspace{-0.2cm}&:&\hspace{-0.2cm} u\in C^{\alpha}(J,\mathbb R), \,v\in C^{\alpha}(J,\mathbb C),\, u>0,\, 0\leqslant|v|\leqslant u,\, |\log u|_{C^{\alpha}}\leqslant C_0 |b|^{\alpha}\\
&& \hspace{2.5cm} |v(x)-v(y)|\leqslant C_0 |b|^{\alpha}u(y)d(x,y)^{\alpha}\,\,\,\text{for all}\,\,\, x,y\in J\Big\},
\end{eqnarray}
where the constant
\begin{equation}\label{defC0}
C_0:=4|f^{-1}_0|_{\infty}|f_0|_{C^{\alpha}(I)}+2(|\phi|_{C^{\alpha}(I)}+|r|_{C^{\alpha}(I)})\cdot(1-\lambda^{-\alpha})>0
\end{equation}
will be referred to as the amplitude of the cone.

\medskip
The next lemma guarantees the invariance of this cone, and contraction of its functions, under the action of the twisted transfer operators. 
At this point we may need to reduce the constant $\delta$ (and make necessary adjustments to $\Delta$ and $n_2$). 
Recall that $\delta>0$ was assumed to satisfy ~\eqref{eq:def-consts0} and $\eta_0=\frac{1}{2}(\sqrt{7}-1)$. 
Diminishing $\delta$, if
necessary, we assume further that
\begin{eqnarray}\label{cancel1}
C_0\delta^{\alpha}<\tfrac{1}{6},\,\;\,\tfrac{2}{3}e^{C_0\delta^{\alpha}}<\eta_0\quad  \text{and}\quad C_7\delta<\tfrac{\pi}{6}.
\end{eqnarray}
Let $\Delta=4\pi/\delta C_7>0$ and $b_0>1$ (depending on $\delta$) be given by 
Proposition~\ref{prop:transv-implies-wUNI}. Recall that $n=n_1+n_2$ where 
$n_2$ satisfies ~\eqref{eq:def-consts0} and $n_1\geqslant 1$ remained to be chosen. 
From now on, we fix $n_1:=\lfloor\beta\log|b_1|\rfloor$, where $b_1\geqslant b_0>1$ satisfies 
\begin{eqnarray}\label{defxi}
	& C_0\lambda^{-\alpha \lfloor \beta \log b_1\rfloor }\left(\tfrac{8\pi}{\delta C_7}\right)^{\alpha}
	\leqslant\tfrac{1}{4}\left(2-2\cos\left(\tfrac{\pi}{12}\right)\right)^{1/2}
	\leqslant\tfrac{1}{4},\\
	& (2+C_0)\lambda^{-\alpha \lfloor \beta \log b_1\rfloor}<(|\phi|_{C^{\alpha}(J)}+|r|_{C^{\alpha}(J)})(1-\lambda^{-\alpha}),\label{defxi1}\\
	& C_0\lambda^{-\alpha \lfloor \beta \log b_1\rfloor}<1, \label{defxi2}
\end{eqnarray}
and $\beta>0$ will be made explicit in \eqref{eq:xi}.

\medskip 
Throughout the remainder of this section we assume that $T$ is a one-dimensional piecewise-expanding map and $J$ is the interval. 

\begin{lemma}\label{lemm:cancel-N} 
Assume that $J$ is the interval.
For any $s=\sigma+ib,\,\,|\sigma|<\varepsilon,\,\,|b|>b_1$ and every $(u,v)\in\cC_b$ the following holds:
if $x_0\in J$ satisfies $\mathcal{D}^{n_2}(T^{n_1}h_{\overline{\omega}}x_0)\cC\pitchfork\mathcal{D}^{n_2}(T^{n_1}h_{\omega}x_0)\cC$ for some $\omega,\overline\omega \in\mathcal{P}^{(n)}$ then there 
exists $x_1\in B(x_0,\Delta/|b|)\cap \mbox{Dom}(\psi_{\overline{\omega},\omega})$ such that either
\begin{enumerate}
\item[(a)]
$|A_{s,h_{\omega},n}(v)+A_{s,h_{\overline{\omega}},n}(v)|\leqslant \eta_0 A_{\sigma,h_{\omega},n}(u)+A_{\sigma,h_{\overline{\omega}},n}(u)$,
\item[(b)]
$|A_{s,h_{\omega},n}(v)+A_{s,h_{\overline{\omega}},n}(v)|\leqslant A_{\sigma,h_{\omega},n}(u)+ \eta_0A_{\sigma,h_{\overline{\omega}},n}(u)$,
\end{enumerate}
for every $x\in B(x_1,\delta/|b|)\cap \mbox{Dom}(\psi_{\overline{\omega},\omega})$.
\end{lemma}

\begin{proof}
Since $\|(DT^k(x))^{-1}\|\leqslant\lambda^{-k}$ for all $x\in J$ and $k\in\mathbb{N}$ we have
\begin{eqnarray}\label{cancel2}
d(h_{\omega}x,h_{\omega}x_0)
	\leqslant\lambda^{-n} d(x,x_0)
	\leqslant\lambda^{-n}\delta / |b|,
	\quad \forall x\in B(x_0, \delta / |b|) \cap \mbox{Dom}(\psi_{\overline{\omega}, \omega}).
\end{eqnarray}
Using that $(u,v)\in\cC_b$ together with inequalities \eqref{cancel1} and \eqref{cancel2} we get
\begin{eqnarray}\label{cancel3}
|v(h_{\omega}x)-v(h_{\omega}x_0)|
	&\leqslant & C_0|b|^{\alpha} \,u(h_{\omega}x_0) \,d(h_{\omega}x,h_{\omega}x_0)^{\alpha}\nonumber\\ 
	&\leqslant &C_0\lambda^{-\alpha n} \, u(h_{\omega}x_0) \, \delta^{\alpha} 
	\,<\,\tfrac{1}{6}u(h_{\omega}x_0)
\end{eqnarray}
for every $x\in B(x_0,\delta/|b|)\cap \mbox{Dom}(\psi_{\overline{\omega}, \omega})$.
Using again that $(u,v) \in \cC_b$ and \eqref{cancel2},
\begin{eqnarray*}
|\log u(h_{\omega}x)-\log u(h_{\omega}x_0)| 
	\leqslant C_0|b|^{\alpha}d(h_{\omega}x,h_{\omega}x_0)^{\alpha}
	\leqslant {C_0\lambda^{-\alpha n}\delta^{\alpha}}
\end{eqnarray*}
and, consequently, 
\begin{eqnarray}\label{cancel4}
\tfrac{2}{3}u(h_{\omega}x_0)
	\leqslant \tfrac{2}{3}e^{C_0\lambda^{-\alpha n}\delta^{\alpha}}u(h_{\omega}x)
	< \eta_0 \, u(h_{\omega}x)
\end{eqnarray}
for every $x\in B(x_0,\delta/|b|)\cap \mbox{Dom}(\psi_{\overline{\omega}, \omega})$.
It will be useful to control the oscillation of $v$ in some ball $B(x_1,\delta/|b|)$ around some point $x_1\in B(x_0,\Delta/|b|)$. 
For that purpose, we estimate the oscillation of $v$ in the ball $B(x_0,(\Delta+\delta)/|b|)$.
Set $\xi:=\Delta+\delta \in (0, \frac{8\pi}{\delta C_7})$.
Then, the choice of constants in ~\eqref{defxi} ensures that
\begin{align}\label{cancel5}
|v(h_{\omega}x)-v(h_{\omega}x_0)|
	& \leqslant C_0|b|^{\alpha} \,u(h_{\omega}x_0) \,d(h_{\omega}x,h_{\omega}x_0)^{\alpha} \nonumber \\
	& \leqslant C_0 \, \lambda^{-\alpha n} \, \xi^{\alpha} \, u(h_{\omega}x_0)  \nonumber \\
	& \leqslant\tfrac{1}{4}\left(2-2\cos\left(\tfrac{\pi}{12}\right)\right)^{1/2}u(h_{\omega}x_0)
	\leqslant \tfrac{1}{4}u(h_{\omega}x_0)
\end{align}
for all $x\in B(x_0,\xi/|b|) \cap \mbox{Dom}(\psi_{\overline{\omega}, \omega})$.

\medskip
At this point, we split the proof in two cases. In the first case we suppose that $|v(h_{\omega}x_0)|\leqslant\frac{1}{2}u(h_{\omega}x_0)$ (if $|v(h_{\overline{\omega}}x_0)|\leqslant\frac{1}{2}u(h_{\overline{\omega}}x_0)$ the situation is analogous). In such case,  inequalities
\eqref{cancel3} and \eqref{cancel4} imply that 
\begin{eqnarray*}
|v(h_{\omega}x)|&\leqslant &|v(h_{\omega}x_0)|+|v(h_{\omega}x)-v(h_{\omega}x_0)|\\
&\leqslant &\tfrac{1}{2}u(h_{\omega}x_0)+{\tfrac{1}{6}} u(h_{\omega}x_0)=\tfrac{2}{3}u(h_{\omega}x_0)\leqslant\eta_0 \,u(h_{\omega}x)
\end{eqnarray*}
for every {$x\in B(x_0,\delta/|b|)\cap \mbox{Dom}(\psi_{\overline{\omega}, \omega})$}.
This implies that $|A_{s,h_{\omega},n}(v)|\leqslant \eta_0 A_{\sigma,h_{\omega},n}(u)$ on $B(x_0,\delta/|b|)\cap \mbox{Dom}(\psi_{\overline{\omega}, \omega})$, thus proving that the first alternative in the lemma holds with $x_1=x_0$. 

\smallskip
The second case occurs when $|v(h_{\omega}x_0)|>\frac{1}{2}u(h_{\omega}x_0)$ and $|v(h_{\overline{\omega}}x_0)|>\frac{1}{2}u(h_{\overline{\omega}}x_0)$.
Using polar coordinates we write 
$$
A_{s,h_{\omega},n}(v)(x)=B_{\omega}(x)e^{i\theta_{\omega}(x)}, 
	\qquad A_{s,h_{\overline{\omega}},n}(v)(x)=B_{\overline{\omega}}(x)e^{i\theta_{\overline{\omega}}(x)}
$$ 
and 
$\theta(x):=\theta_{\omega}(x)-\theta_{\overline{\omega}}(x)$. Following \cite[Lemma 2.9]{AM} it is sufficient to show that
there exists $x_1\in B(x_0,\Delta/|b|)\cap \mbox{Dom}(\psi_{\overline{\omega}, \omega})$ so that $|\theta(x)-\pi|\leqslant 2\pi/3$ for all $x\in B(x_1,\delta/|b|)\cap \mbox{Dom}(\psi_{\overline{\omega}, \omega})$.
For that purpose, it is enough to show that $|\theta(x)-\pi|\leqslant 2\pi/3$ for all $x\in B(x_0,\xi/|b|)\cap \mbox{Dom}(\psi_{\overline{\omega}, \omega})$.
Observe that $\theta=V-b\, \psi_{\omega,\overline{\omega}}$, where $V=\arg(v\circ h_{\omega})-\arg(v\circ h_{\overline{\omega}})$. In order to estimate $V$ we shall use the following basic fact present on \cite[Lemma 2.9]{AM}: if 
$z_1,z_2\in \mathbb C$, $|z_1|, |z_2|\geqslant c$ and $|z_1-z_2|\leqslant c(2-2\cos z)^{1/2}$ for some $c>0$ and $|z|<\pi$ then $|\arg(z_1)-\arg(z_2)|\leqslant z$. It follows from \eqref{cancel5}  that 
\begin{eqnarray}\label{cancel6}
|v(h_{\omega}x)-v(h_{\omega}x_0)|\leqslant\tfrac{1}{4} \, (2-2\cos\tfrac{\pi}{12})^{1/2}\, u(h_{\omega}x_0),
\end{eqnarray}
and similarly for $h_{\overline{\omega}}$.
We also have
\begin{eqnarray}\label{cancel7}
|v(h_{\omega}x)|&\geqslant &|v(h_{\omega}x_0)|-|v(h_{\omega}x)-v(h_{\omega}x_0)|\nonumber\\
&\geqslant &\tfrac{1}{2} u(h_{\omega}x_0)-\tfrac{1}{4} u(h_{\omega}x_0)=\tfrac{1}{4} u(h_{\omega}x_0),
\end{eqnarray}
and equally for $h_{\overline{\omega}}$.
Then, \eqref{cancel6} and \eqref{cancel7} imply that 
\begin{eqnarray}\label{cancel8}
|\arg(v(h_{\omega}x))-\arg(v(h_{\overline{\omega}}x_0))|\leqslant \pi/12,
	\quad\text{and so}\quad
	|V(x)-V(x_0)|\leqslant \pi/6.
\end{eqnarray}
Now, the tranversality condition $\mathcal{D}^{n_2}(T^{n_1}h_{\overline{\omega}}x_0)\cC\pitchfork\mathcal{D}^{n_2}(T^{n_1}h_{\omega}x_0)\cC$ at $x_0$ implies on the UNI condition stated at Proposition \ref{prop:transv-implies-wUNI}.
As $T$ acts on the interval we have that $\|D\psi_{\omega,\overline{\omega}}(z)^{-1}\|^{-1}=\|D\psi_{\omega,\overline{\omega}}(z)\|$. Hence, 
the mean value inequality
implies that 
\begin{eqnarray*}
|b(\psi(x)-\psi(x_0))|
	\geqslant |b| \, \inf\|D\psi_{\omega,\overline{\omega}}(z)^{-1}\|^{-1} \,d(x,x_0)
	\geqslant  \tfrac{C_7}{2}|b|  \rho^{-n_2}   \,d(x,x_0)
\end{eqnarray*}
for every $x\in B(x_0,\Delta/|b|)\cap \mbox{Dom}(\psi_{\overline{\omega}, \omega})$
(here the infimum is taken over $B(x_0,\Delta/|b|)\cap \mbox{Dom}(\psi_{\overline{\omega}, \omega})$).
Recall $n_2\geqslant 1$ was chosen so that $\rho^{-n_2}=\delta $. 
Then, as $\Delta=4\pi/\delta C_7$ and $B(x_0,\Delta/|b|)\cap \mbox{Dom}(\psi_{\overline{\omega}, \omega})$ contains an interval of length at least $\Delta/|b|$ 
we conclude that $b(\psi(x)-\psi(x_0))$ fills out an interval around $0$ of length at least $2\pi$ as $x$ varies in $B(x_0,\Delta/|b|)\cap \mbox{Dom}(\psi_{\overline{\omega}, \omega})$. This allow us to choose $x_1\in B(x_0,\Delta/|b|)\cap \mbox{Dom}(\psi_{\overline{\omega}, \omega})$ so that
\begin{eqnarray*}
b(\psi(x_1)-\psi(x_0))=\theta(x_0)-\pi \; (\mbox{mod} 2\pi),
\end{eqnarray*}
and so
\begin{eqnarray*}
V(x_1)-V(x_0)=V(x_1)-b\psi(x_1)-\pi+\theta(x_0)-V(x_0)+b\psi(x_0)=\theta(x_1)-\pi.
\end{eqnarray*}
In this way, it follows from \eqref{cancel8} that $|\theta(x_1)-\pi|\leqslant\pi/6$. Finally, using the choice of $\delta$ in 
\eqref{cancel1}, the mean value theorem and that 
$|\psi_{\overline{\omega}, \omega}'|_{\infty}\leqslant C_7$ (as a consequence of Lemma~\ref{lemma1}),
we obtain
\begin{eqnarray*}
|\theta(x)-\pi| & \leqslant &|\theta(x_1)-\pi| + |\theta(x)-\theta(x_1)|\\
&\leqslant &\pi/6+|\theta(x)-\theta(x_1)|\\
&\leqslant &\pi/6+|b||\psi(x)-\psi(x_1)|+|V(x)-V(x_0)|+|V(x_1)-V(x_0)|\\
&\leqslant &\pi/6+C_7\delta+\pi/6+\pi/6\leqslant 2\pi/3
\end{eqnarray*}
for all $x\in B(x_1,\delta/|b|)\cap \mbox{Dom}(\psi_{\overline{\omega}, \omega})$, as desired.
\end{proof}

The previous lemma, which makes strong use of the interval, ensures that for each pair $(u,v)\in\cC_b$ and $x\in I$ there exist inverse branches $h_{\omega_x}$ and $h_{\overline{\omega}_x}$, parameterized by $\omega_x, \overline{\omega}_x\in \mathcal P^{(n)}$, 
such that the sum of the corresponding terms appearing in the twisted transfer operator exhibit cancellations on the set 
$$
B(x,\delta/|b|) \cap \mbox{Dom}(\psi_{\overline{\omega}_x, \omega_x}) 
\subset B(x,\Delta/|b|)\cap \mbox{Dom}(\psi_{\overline{\omega}_x, \omega_x}).
$$
Recall that $\mbox{Dom}(\psi_{\overline{\omega}_x, \omega_x}) = T^n(\overline{\omega}_x) \cap T^n(\omega_x)$
is the intersection of two intervals obtained as images of elements in the Markov partition $\mathcal P$, hence there are finitely many such 
intervals. 
Actually, by Markov property,
$
\{T(\overline{\omega})\cap T(\omega): \overline{\omega}, \omega\in\mathcal{P}\} 
	=\{T^k(\overline{\omega}) \cap T^k(\omega): \overline{\omega}, \omega\in\mathcal{P}^{(k)}\}
$
for every $k\geqslant 1$, and every element in this collection is a finite union of elements of $\mathcal P$. 
Let $\mathcal{I}$ denote the partition of the interval determined by 
these family of sets, which is coarser than $\mathcal P$.
In particular, there exists $b_2 \geqslant b_1$ so that 
\begin{eqnarray}\label{defb2}
\frac{2\Delta}{b_2}\leqslant \diam (\mathcal I)\leqslant \min_{\omega,\overline\omega \in\mathcal{P}^{(k)}}\diam \{\mbox{Dom}(\psi_{\overline{\omega}, \omega})\}
	\quad\text{for every $k\geqslant 1$.}
\end{eqnarray}

\medskip
Now, given $|b|\geqslant b_2$, consider the partition $\mathcal{Q}=\mathcal{Q}_b$, obtained as refinement of $\mathcal P$ given by Lemma~\ref{le:partQ}. 
in a way that $Q\in \mathcal Q$ satisfies $\frac{2\Delta}{|b|}\leqslant\diam Q\leqslant \frac{2\Delta}{|b|}\rho$. Write $\mathcal{Q}=(Q_j)_{1\leqslant j \leqslant \kappa}$, for some $\kappa \leqslant |b|/2\Delta$.
We claim that 
for each  $1 \leqslant j \leqslant \kappa$ there exist $\overline{\omega}_j=\overline{\omega}_{x_j}, \omega_j=\omega_{x_j} \in\mathcal{P}^{(n)}$ 
and a ball $B_j=B(x_j,{\delta}/{2|b|})\subset Q_j\cap T^n(\omega)\cap T^n(\overline{\omega})$ 
where we observe cancellations. 
Indeed, as $\diam Q_j\geqslant {2\Delta}/{|b|}$, Lemma~\ref{lemm:cancel-N} ensures that there exists $z_j\in Q_j$ and inverse branches $h_{\omega_{z_j}}$ and $h_{\overline{\omega}_{z_j}}$ so that either 
$$
|A_{s,h_{\omega_{z_j}},n}(v)+A_{s,h_{\overline{\omega}_{z_j}},n}(v)|\leqslant \eta_0 A_{\sigma,h_{\omega_{z_j}},n}(u)+A_{\sigma,h_{\overline{\omega}_{z_j}},n}(u)
$$
or 
$$
|A_{s,h_{\omega_{z_j}},n}(v)+A_{s,h_{\overline{\omega}_{z_j}},n}(v)|\leqslant A_{\sigma,h_{\omega_{z_j}},n}(u)+ \eta_0A_{\sigma,h_{\overline{\omega}_{z_j}},n}(u)
$$
on the set $B(z_j,\delta/|b|)\cap T^n(\omega_{z_j})\cap T^n(\overline{\omega}_{z_j})$. 
By \eqref{defb2} and the one-dimensionality of the ambient space, the previous intersection contains a ball $B_j(x_j,\delta/2|b|)$ and the claim follows taking 
$\omega_{x_j}=\omega_{z_j}$ and ${\overline{\omega}_{x_j}}={\overline{\omega}_{z_j}}$. We will use the notation
\begin{equation}\label{omegawinners}
\tilde \omega_j =
\begin{cases}
\begin{array}{ll}
\omega_{x_j}, & 
	\text{if}\; 
|A_{s,h_{\omega_{x_j}},n}(v)+A_{s,h_{\overline{\omega}_{x_j}},n}(v)|\leqslant \eta_0 A_{\sigma,h_{\omega_{x_j}},n}(u)+A_{\sigma,h_{\overline{\omega}_{x_j}},n}(u) \\
\smallskip
\overline{\omega}_{x_j}, & \text{if} \; 
|A_{s,h_{\omega_{x_j}},n}(v)+A_{s,h_{\overline{\omega}_{x_j}},n}(v)|\leqslant A_{\sigma,h_{\omega_{x_j}},n}(u)+ \eta_0A_{\sigma,h_{\overline{\omega}_{x_j}},n}(u),
\end{array}
\end{cases}
\end{equation}
and denote by $\hat \omega_j$ the remaining interval.

\medskip
We proceed to construct a family of bump functions adapted to the previous family of intervals. Indeed, for each $(u,v)\in\cC_b$, $n\geqslant 1$ and $\eta\in [\eta_0,1]$ there exists a $C^1$-bump function 
\begin{equation}\label{def-xi}
\chi=\chi(b,u,v,n,\eta): I \to [\eta,1]
\end{equation} 
satisfying:
\begin{enumerate}
\item $\chi (h_{\tilde \omega_j}(x))=\eta$, for every $x\in B(x_j, {\delta}/{6|b|})$ and $j=1,2, \dots \kappa$
	\smallskip
\item $\chi (x)=1$, for every $x\notin \bigcup_{j=1}^\kappa  h_{\tilde \omega_j}(B(x_j, \delta/2|b|)) $ 
	\smallskip
\item ${|\chi'|\leqslant\dfrac{6(1-\eta)|b|}{\Delta M}}$, where $M=\displaystyle\min_{\omega}\{\inf|h'_{\omega}|\}$.
\end{enumerate}
Indeed, item (3) can be obtained by the mean value theorem since the diameter of the image by $h_{\omega}$ of each sub-interval adjacent to the middle-third one is bounded below by  $\inf |h_{\omega}'| \cdot \frac{\delta}{6|b|}$.  
Finally, throughout we fix $\eta$ sufficiently close to one (and independent of the functions $u$ and $v$) 
so that $|\chi'| \leqslant |b|$ (see Figure~1 below).

\begin{figure}[!htb]
    \label{construcaoCHI}
    \subfloat[\label{fig:seu_madruga} Inside $Q_j\in \mathcal Q$ (in purple) there exists a ball $B(x_j,\delta/2|b|)$ (in red) and two inverse branches (in green and brown) 
    which exhibit cancelations.]{
        \includegraphics[width=0.48\textwidth]{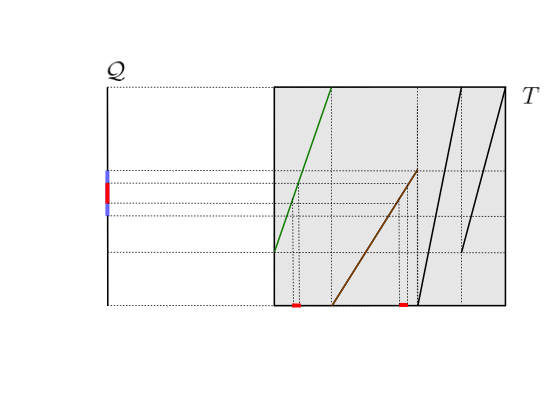}
    }\hfill
    \subfloat[\label{fig:seu_barriga} For each $Q_j\in \mathcal Q$, the image of the ball $B(x_j, \frac\delta{6|b|})$ (in red in the vertical axis) by the inverse branch determined by the winning interval $\tilde \omega_j\in \mathcal P^{(n)}$ determines an interval (in red in the horizontal axis) where the bump function $\chi$ 
    takes value $\eta$.
    ]{
        \includegraphics[width=0.48\textwidth]{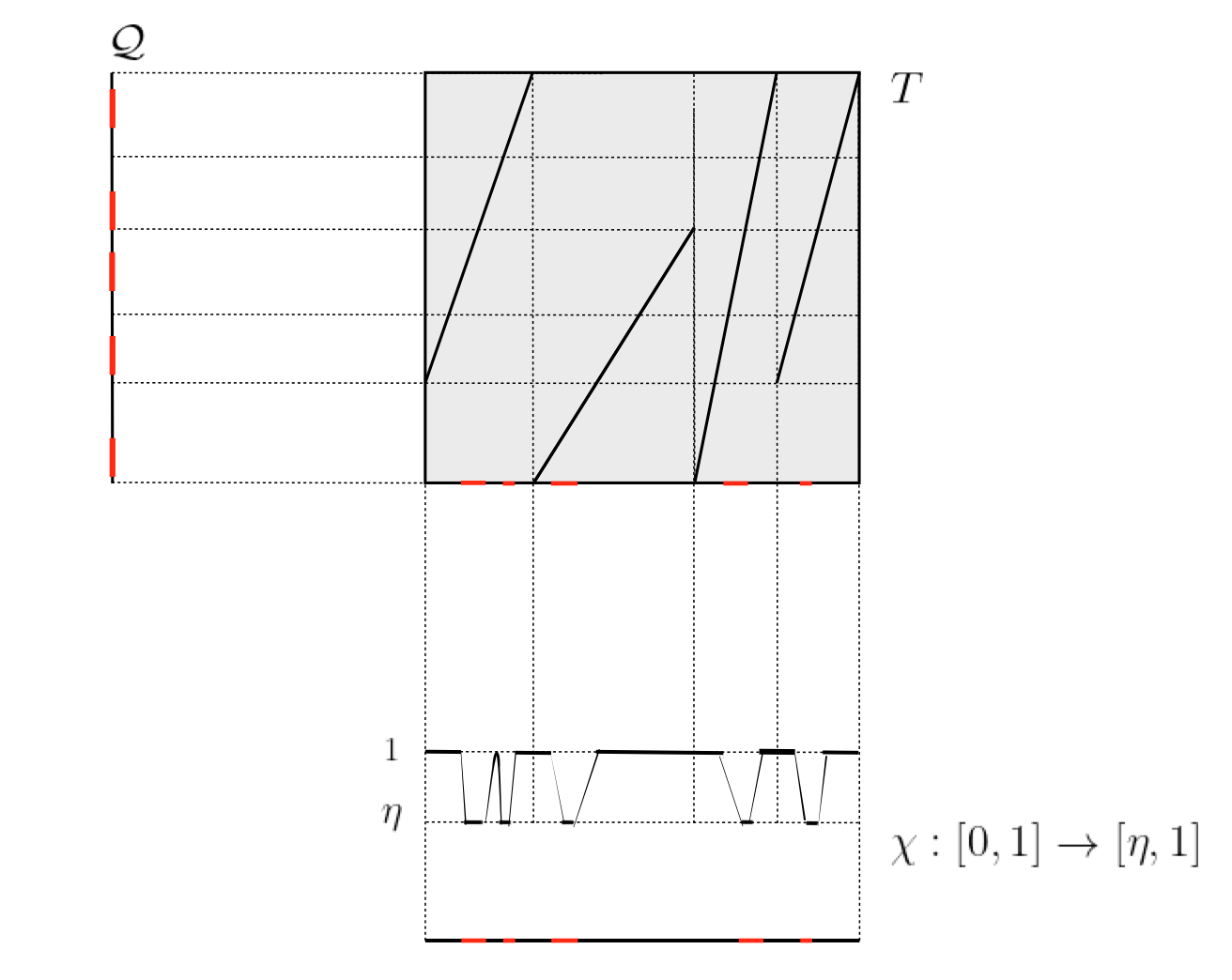}
    }
        \caption{Construction of the bump function}
\end{figure}

\section{Decay of correlations for suspension semiflows}\label{sec:flows-decay-transversal}

The main goal of this section is to prove the dichotomy in Theorem~\ref{thm:A}. As observed in Subsection~\ref{geom-transv}, the roof function is not cohomologous to a piecewise constant one if and only if it satisfies the transversality condition. Hence, we are left to prove exponential decay of correlations for Gibbs measures associated to suspension semiflows over such roof functions. 
Throughout this section $T$ is a piecewise expanding Markov map on the interval $I$.

\subsection{Two reductions}

The argument adapts the classical strategy developed by Dolgopyat~\cite{D} of reducing the proof of 
the exponential decay of correlations to a contraction of the twisted transfer operators. 
The key estimate will be the following contraction of $L^1$-norms of the twisted transfer operators with respect to the $L^{\infty}$-norm.
Recall that the constants
\begin{eqnarray}\label{eq:cts}
N,\;\; \delta,\;\; \Delta, \;\; n_2,\;\;  n_1=\lfloor \beta\log |b_1|\rfloor \;\; \text{and} \;\; n=n_1+n_2,  
\end{eqnarray}
where chosen in this order, where $\beta>0$ was to be determined.   
We now choose $\beta>0$ small so that\footnote{As $\lambda^{\frac{\alpha \lfloor \beta \log_{\lambda}|b|\rfloor}{16}} \leqslant C_0 |b|^\alpha$ is equivalent to $\beta\leqslant 16(1+\log C_0/\alpha\log|b|)$ we can choose $b_2$ large enough in a way that $\beta=\beta(C_0,\alpha)$ can be taken independently of $b$.} 
\begin{eqnarray}\label{eq:xi}
\lambda^{\frac{\alpha \lfloor \beta \log b\rfloor}{16}} \leqslant C_0 |b|^\alpha, 
	\quad\text{for every $|b|>b_2$.}
\end{eqnarray}

\begin{proposition}\label{mainprop}
Let $T:I\to I$ be a $C^{1+\alpha}$ uniformly expanding Markov map and $r:I\to\mathbb{R}^{+}$ 
be a piecewise $C^{1+\alpha}$ smooth roof function satisfying (H1) and (H2).
If there exists no piecewise $C^1$-smooth function $\theta: I \to\mathbb R$ 
such that $\chi=r-\theta\circ T+\theta$ is constant on each element of the partition $\cP$
then there exist $\vep>0$, $b_3>b_2$  and $\xi>0$   so that, for each $s=\sigma+ib$ so that $|\sigma|<\vep,\,|b|>b_3$ and $k=\lfloor \beta\log |b|\rfloor$, we have
\begin{eqnarray*}
\|{\mathscr L}^k_sv\|_{L^1(\mu_{\sigma})}\leqslant \lambda^{-\xi k}\|v\|_{L^{\infty}(\mu_{\sigma})}
\end{eqnarray*}
for every observable $v\in C^{\alpha}(I)$ such that  $\|v\|_{(b)}< \lambda^{\frac{\alpha k}{16}} \|v\|_{L^{\infty}(\mu_{\sigma})}$.
\end{proposition}

Let us first derive an implication of this proposition while postponing its proof to Subsection~\ref{subsec:cancel-trans}.
Indeed, proceeding as in \cite[Subsection~2.3]{DV}, Theorem~\ref{thm:A} can be reduced to the proof of the following:

\begin{theorem}\label{mainthm0}
Let $T:I\to I$ be a $C^{1+\alpha}$ piecewise expanding Markov interval map, let $r:I\to\mathbb{R}^{+}$  be a piecewise 
$C^{1+\alpha}$ smooth roof function satisfying (H1) and (H2), and suppose that there exists no piecewise smooth map $\theta: I\to\mathbb{R}$ so that $r=\theta\circ T-\theta+\chi$ with $\chi$ constant on partition elements.
There are constants $b_5\geqslant b_3$, $\zeta\in (0,1)$ and $B>0$ so that
\begin{eqnarray*}
\|{\mathscr L}^\ell_s\|_{(b)}\leqslant \zeta^\ell
\end{eqnarray*}
 for all $s=\sigma+ib,\,\sigma\geqslant -\vep,\,|b|\geqslant b_5$ and every $\ell \geqslant B\log |b|$.
\end{theorem}

\begin{proof}
Let us deal separately with functions having large or small oscillation.
If an observable $v\in C^\alpha(I)$ satisfies
 $\|v\|_{L^{\infty}(\mu_{\sigma})}\leqslant\lambda^{-\frac{\alpha k}{16}}\|v\|_{(b)}$  for $k=\lfloor\beta\log|b|\rfloor$, the Lasota-Yorke's inequality on 
Proposition~\ref{L-Y0} implies that
\begin{eqnarray}
\|{\mathscr L}^k_sv\|_{(b)}&\leqslant & C_8\left(\lambda^{-\alpha n}\|v\|_{(b)}+\|v\|_{L^{\infty}(\mu_{\sigma})}\right) \nonumber\\
&\leqslant &C_8(\lambda^{-\alpha k}+\lambda^{-\frac{\alpha k}{16}})\|v\|_{(b)}. \label{est1f}
\end{eqnarray}
If alternatively  $\|v\|_{(b)}< \lambda^{\frac{\alpha k}{16}} \|v\|_{L^{\infty}(\mu_{\sigma})}$ 
then 
$ 
\| {\mathscr L}^k_sv\|_{L^1(\mu_{\sigma})}
 \leqslant \lambda^{-\xi k}\|v\|_{(b)}$ by Proposition~\ref{mainprop}.
Moreover, 
applying the Lasota-Yorke inequality twice and using  
$\|\cdot\|_{L^{\infty}(\mu_{\sigma})}\leqslant \|\cdot\|_{(b)}$ we get 
\begin{eqnarray}\label{estimateLY2}
\| {\mathscr L}^{2k}_sv\|_{(b)} 
	\leqslant 2C_8^2\lambda^{-\alpha k}\|v\|_{(b)}
	+C_8  
\| {\mathscr L}^{k}_sv\|_{L^{\infty}(\mu_{\sigma})}. 
\end{eqnarray}
In this point we shall use 
the interpolation result present in \cite[Lemma~A.4]{BW} (with $d=u=1$) which guarantees that there are $C_{10},\varepsilon_0>0$ such that
\begin{eqnarray*}\label{interpolation}
\|v\|_{L^{\infty}(\mu_{\sigma})}\leqslant C_{10}\varepsilon^{-1}\|v\|_{L^1(\mu_{\sigma})}+\varepsilon^\alpha |v|_{C^{\alpha}(\mu_{\sigma})}, 
	\quad \forall \varepsilon\in (0,\varepsilon_0).
\end{eqnarray*}
Altogether, inequality \eqref{estimateLY2}, the previous interpolation (for $\varepsilon=\lambda^{-\frac{\xi}{8}k}$) 
and the contraction of the $L^1$-norm 
$ \| {\mathscr L}^k_sv\|_{L^1(\mu_{\sigma})}
 \leqslant \lambda^{-\xi k}\|v\|_{(b)}$ imply that
\begin{align}
\| {\mathscr L}^{2k}_sv\|_{(b)} 
	& \leqslant \left(2C^2_8\lambda^{-\alpha k}+2  C_{10}C_8\lambda^{-\frac{\alpha\xi}{16}k}\right)\|v\|_{(b)}+C_8C_{10}\lambda^{\frac{\xi k}{8}}\|{\mathscr L}^k_sv\|_{L^{1}(\mu_{\sigma})} \nonumber\\
	&  \leqslant \left(2C_{10}C^2_8\lambda^{-\alpha k}+2 C_{10}C_8\lambda^{-\frac{\alpha\xi}{16}k} +C_8C_{10}\lambda^{-\frac{7\xi k}{8}} \right)\|v\|_{(b)}. \label{est2f}
\end{align}
Summoning the two estimates ~\eqref{est1f} and \eqref{est2f}, and choosing a large $b_5\geqslant b_3$, we conclude that
there exists $\tilde \zeta\in (0,1)$ so that for every observable $v\in C^\alpha(I)$
either
$
\|{\mathscr L}^{k}_s v \|_{(b)} \leqslant \tilde \zeta^k\|v\|_{(b)}
$
or 
$
\|{\mathscr L}^{2k}_s v \|_{(b)} \leqslant \tilde \zeta^{2k}\|v\|_{(b)}
$
 for all $s=\sigma+ib,\,\sigma\geqslant -\vep$ and $|b|\geqslant b_5$, where  $k=\lfloor\beta\log|b|\rfloor$. 
 
 This is enough to conclude the proof of the theorem. 
 Indeed, fix an arbitrary $v\in C^\alpha(I)$. If $\ell\geqslant 2\lfloor \beta \log |b|\rfloor  \geqslant 2n$ then by integer division one can write
$\ell=\sum_{i=1}^{q} \ell_i +r$ for some $q\geqslant 1$ and $0\leqslant r<2n$, where $v_0=v$, $v_{i+1}={\mathscr L}^{\ell_i}_s v_i$
for each $0\leqslant i < \ell$ and
$$
\ell_i=
\begin{cases}
\begin{array}{ll}
k &, \text{if } \|v\|_{L^{\infty}(\mu_{\sigma})}\leqslant\lambda^{-\frac{\xi\alpha k}{16}}\|v\|_{(b)}\\
2k &, \text{otherwise.}   
\end{array}
\end{cases}
$$
Then, it follows that
\begin{eqnarray*}
\|{\mathscr L}^\ell_s v\|_{(b)}
	\leqslant \|{\mathscr L}^r_s\|_{(b)} \; \prod_{i=0}^{q-1} \|v_i\|_{(b)} 
	\leqslant \Big(\frac{\|{\mathscr L}_s\|_{(b)}}{\tilde \zeta}\Big)^{2n}\; \tilde \zeta^{\ell} \; \|v\|_{(b)} 
	\leqslant \zeta^\ell \; \|v\|_{(b)}
\end{eqnarray*}
if one takes $\zeta=\tilde \zeta^{1/2}$ and $\ell\geqslant B\log |b|$ for some $B\gg \beta$.
This completes the proof of the theorem.
\end{proof}

\subsection{Contraction of the $L^1$-norm}\label{subsec:cancel-trans} 
In this subsection we prove Proposition~\ref{mainprop}. Actually we will show how to use the notion of transversality to establish contraction in $L^2$~-~norm for iterates of twisted transfer operator when applied over suitable functions with some controlled oscillation and then apply H\"older's inequality to obtain the result.

{Recall that $\delta>0$ satisfies \eqref{eq:def-consts0} and ~\eqref{cancel1}, $\Delta>0$ (depending on $\delta$) was determined in Proposition~\ref{prop:transv-implies-wUNI} and $b_2$ was defined in \eqref{defb2}, and $n\geqslant 1$ be given by ~\eqref{eq:def-consts}.
The starting point is the following standard consequence of the cancellations in Lemma~\ref{lemm:cancel-N}}  (see e.g. \cite[Lemma~2.7]{BV}). 
We have the following initial estimate for functions in the cone $\mathcal{C}_b$ defined by \eqref{def:coneCb} with a constant $C_0$ defined in ~\eqref{defC0}.

\begin{corollary}\label{corollarycancel}
If {$|b|> b_2$}, $(u,v)\in\mathcal{C}_b$ and 
$\chi=\chi(u,v)$ is the bump function described in \eqref{def-xi} then 
$|\mathscr L^{n}_sv|\leqslant \mathscr L^{n}_{\sigma}(\chi u)$ 
for every $s=\sigma+ib$ so that\, $|\sigma|<\varepsilon$.
\end{corollary}
\begin{proof} 
This is a simple consequence of the cancellations in  \eqref{omegawinners} and definition of $\chi$ in ~\eqref{def-xi}. As $(u,v)\in\mathcal{C}_b$ then $0\leqslant|v|\leqslant u$. Moreover, using the previous terminology, if $x\in I$ does not belong to any ball $B(x_j, \delta/2|b|)$ then $\chi(h_\omega(x))=1$ for every $\omega\in \mathcal P^{(n)}$, and so $|\mathscr L^{n}_sv(x)|\leqslant \mathscr L^{n}_{\sigma}(u)(x) = \mathscr L^{n}_{\sigma}(\chi u)(x)$.
Otherwise, $x\in B(x_j, \delta/2|b|)$ for some $1\leqslant j \leqslant \kappa$ and
\begin{eqnarray*}
|\mathscr L^{n}_sv(x)|
	\leqslant\frac{1}{\lambda^n_{\sigma}f_{\sigma}(x)}
		[\eta_0 A_{\sigma,h_{\widetilde{\omega}_x,n}}(u)(x)
		+ A_{\sigma,h_{\hat {\omega}_x,n}}(u)(x)
		+\sum_{\substack{\omega\in\mathcal{P}^{(n)} \\ \omega\notin \{\widetilde{\omega}_x,\, \hat{\omega}_x\}}} A_{\sigma,h_{\omega},n}(|v|)(x)]
		\leqslant \mathscr L^{n}_{\sigma}(\chi u)(x).
\end{eqnarray*} 
This proves the corollary.
\end{proof}

The latter is used in the proof of the invariance of the cone condition. 

\begin{lemma}\label{invarcone}
If $s=\sigma+ib\,$ with $|\sigma|<\varepsilon$ and  {$|b|> b_2$}   then 
$
(\mathscr  L^{n}_{\sigma}(\chi u), \mathscr L^{n}_s v)\in \mathcal{C}_b
$
for every pair $(u,v)\in \cC_b$.
\end{lemma}
\begin{proof}
Given $(u,v)\in \cC_b$ let $\chi=\chi(b,u,v,n,\eta)$ be the bump-function given by \eqref{def-xi} and set $u_n:=\mathscr  L^{n}_{\sigma}(\chi u)$ and $v_n:=\mathscr L^{n}_s v$. We will prove that $(u_n, v_n)\in \mathcal{C}_b$. 
By positivity of $\mathscr L_{\sigma}$ and Corollary~\ref{corollarycancel} it is simple to check that $u_n>0$ and $|v_n|\leqslant u_n$.
We claim that $|\log u_n|_{C^{\alpha}}\leqslant C_0 |b|^{\alpha}$. Given $x,y\in I$, we have
\begin{eqnarray*}
\frac{u_n(x)}{u_n(y)}=\frac{f_{\sigma}(y)}{f_{\sigma}(x)}\frac{\sum_{\omega\in\mathcal{P}^{(n)}}A_{\sigma,h_{\omega},n}(f_{\sigma}\chi u)(x)}{\sum_{\omega\in\mathcal{P}^{(n)}}A_{\sigma,h_{\omega},n}(f_{\sigma}\chi u)(y)}
\end{eqnarray*}
and
\begin{eqnarray*}
\left|\frac{(A_{\sigma, h_{\omega}, n}(\chi u))(x)}{(A_{\sigma, h_{\omega}, n}(\chi u))(y)}\right|=\left|\frac{e^{S_n(\phi-\sigma r)(h_{\omega}(x))}}{e^{S_n(\phi-\sigma r)(h_{\omega}(y))}}\frac{f_{\sigma}(h_{\omega}(x))}{f_{\sigma}(h_{\omega}(y))}\frac{\chi(h_{\omega}(x))}{\chi(h_{\omega}(y))}\frac{u(h_{\omega}(x))}{u(h_{\omega}(y))}\right|.
\end{eqnarray*}
We need to estimate each term on the right-hand side above. Firstly, using Remark~\ref{pertubationop} one has $|\log f_{\sigma}|_{C^{\alpha}(I)}\leqslant 4|f^{-1}_0|_{\infty}|f_0|_{C^{\alpha}(I)}$ and, consequently, $\frac{f_{\sigma}(y)}{f_{\sigma}(x)}\leqslant \exp(4|f^{-1}_0|_{\infty}|f_0|_{C^{\alpha}(I)}d(x,y)^{\alpha})$. Now, since $\chi\in [\frac{1}{2},1]$ and $|\chi'|\leqslant |b|$ (hence $|(\log\chi)'|\leqslant 2|b|$ and $|\log\chi(x)-\log\chi(y)|\leqslant\log 2<1$), using the mean value theorem 
we conclude that  
\begin{eqnarray*}
|\log\chi(x)-\log\chi(y)|\leqslant 2\min\{1,|b|d(x,y)\}\leqslant 2|b|^{\alpha}d(x,y)^{\alpha}.
\end{eqnarray*}
Altogether, using $|\log u|_{C^{\alpha}}\leqslant C_0 |b|^{\alpha}$, bounded distortion and that $|\sigma|<\vep$,
\begin{align*}
\left|\frac{(A_{\sigma, h_{\omega}, n}(\chi u))(x)}{(A_{\sigma, h_{\omega}, n}(\chi u))(y)}\right|
	& \leqslant  \exp[(|\phi|_{C^{\alpha}(I)}+|r|_{C^{\alpha}(I)})(1-\lambda^{-\alpha}) \, d(x,y)^{\alpha}] \\
	&  \times \exp[(4|f^{-1}_0|_{\infty}|f_0|_{C^{\alpha}(I)}+2|b|^{\alpha}\lambda^{-n\alpha}) \, d(x,y)^{\alpha}] \\
	& \times\exp[C_0|b|^{\alpha}\lambda^{-n\alpha}\, d(x,y)^{\alpha}]. 
\end{align*}
From the choice of $n_1$ in ~\eqref{defxi1} and definition of $C_0$ in ~\eqref{defC0}, we have that $(2+C_0)\lambda^{-\alpha n}<(|\phi|_{C^{\alpha}(I)}+|r|_{C^{\alpha}(I)})(1-\lambda^{-\alpha})$, and so $|\log u_n|_{C^{\alpha}}\leqslant C_0 |b|^{\alpha}$.
It remains to prove 
\begin{eqnarray}\label{invarcone1}
|v_n(x)-v_n(y)|\leqslant C_0 |b|^{\alpha}u_n(y)d(x,y)^{\alpha}
	\quad\text{for all $x,y\in I$.}
\end{eqnarray}
The argument to show \eqref{invarcone1} is identical to the proof of Proposition \ref{L-Y0}. First of all we write $v_n(x)-v_n(y)=(\mathscr L^n_sv)(x)-(\mathscr L^n_sv)(y)$ as
\begin{align}
& \lambda^{-n}_{\sigma} f_{\sigma}(x)^{-1}
	 \left[\left(\sum_{\omega\in\mathcal{P}^{(n)}}A_{s,h_{\omega},n}(f_{\sigma}v)\right)(x)-\left(\sum_{\omega\in\mathcal{P}^{(n)}}A_{s,h_{\omega},n}(f_{\sigma}v)\right)(y)\right]\label{invarcone2}\\
	&+\Big(\frac{f_{\sigma}(y)}{f_{\sigma}(x)}-1\Big) \lambda^{-n}_{\sigma}f_{\sigma}(y)^{-1}\left(\sum_{\omega\in\mathcal{P}^{(n)}}A_{s,h_{\omega},n}(f_{\sigma}v)\right)(y)\label{invarcone2.1}.
\end{align}
In order to estimate the term \eqref{invarcone2} we write
\begin{align}
A_{s,h_{\omega},n}(f_{\sigma}v)(x)-A_{s,h_{\omega},n}(f_{\sigma}v)(y)
&=e^{-ibS_nr(h_{\omega}x)}(e^{-\sigma S_nr(h_{\omega}x)}-e^{-\sigma S_nr(h_{\omega}y)})(ve^{S_n\phi})(h_{\omega}x)\label{invariance3}\\
&+e^{-\sigma S_nr(h_{\omega}y)}e^{-ib S_nr(h_{\omega}x)} (v(h_{\omega}x)-v(h_{\omega}y)) e^{S_n\phi(h_{\omega}x)} \label{invariance4}\\
&+e^{-\sigma S_nr(h_{\omega}y)}e^{-ib S_nr(h_{\omega}x)}v(h_{\omega}y)(e^{S_n\phi(h_{\omega}x)}-e^{S_n\phi(h_{\omega}y)})\label{invariance5}\\
&+(e^{-ibS_nr(h_{\omega}x)}-e^{-ibS_nr(h_{\omega}y)})e^{S_n(\phi-\sigma r)(h_{\omega}y)}v(h_{\omega}y).\label{invariance6}
\end{align}
Using the same estimates present in Proposition \ref{L-Y0} we conclude that
\begin{eqnarray*}
e^{-ibS_nr(h_{\omega}x)}(e^{-\sigma S_nr(h_{\omega}x)}-e^{-\sigma S_nr(h_{\omega}y)})(ve^{S_n\phi})(h_{\omega}x)\leqslant C_7|\sigma|(e^{S_n(\phi-\sigma r)}|v|)(h_{\omega}x)d(x,y).
\end{eqnarray*}
Recall that $(u,v)\in \cC_b$ and for this reason altogether the choice of $b_1$ we gain
\begin{eqnarray*}
|v(h_{\omega}(x))|&\leqslant & |v(h_{\omega}(y))|+C_0|b|^{\alpha}u(h_{\omega}(y)) \, d(h_{\omega}(x),h_{\omega}(y))^{\alpha}\\
&\leqslant &(1+C_0\lambda^{-n\alpha}|b|^{\alpha}) \, u(h_{\omega}(y))
\leqslant 2|b|^{\alpha}u(h_{\omega}(y)).
\end{eqnarray*}
Hence, the term \eqref{invariance3} is bounded above by
\begin{eqnarray*}
2C_7e^{{n|\phi|}_{C^{\alpha}(I)}}|\sigma||b|^{\alpha}(e^{S_n(\phi-\sigma r)}u)(h_{\omega}y)d(x,y).
\end{eqnarray*}
Similarly, we can use that $|v|_{{C}^{\alpha}(I)}\leqslant C_0|b|^{\alpha}|u|$ to conclude that \eqref{invariance4} is bounded by 
\begin{align*}
&\lambda^{-\alpha n}|v|_{{C}^{\alpha}(I)} {e^{\sigma_0\frac{C_7}{2}\diam(I)}}\,(e^{S_n(\phi-\sigma r)})(h_{\omega}x)\,d(x,y)^{\alpha}\\
&\leqslant\lambda^{-\alpha n}e^{{n|\phi-\sigma r|}_{C^{\alpha}(I)}}C_0{e^{\sigma_0\frac{C_7}{2}\diam(I)}}\,|b|^{\alpha}(e^{S_n(\phi-\sigma r)}u)(h_{\omega}y)\,d(x,y)^{\alpha}.
\end{align*}
Additionally, the expression \eqref{invariance5} is bounded by $\frac{1}{1-\lambda^{-\alpha}} |\phi|_{{C}^{\alpha}(I)} \, (e^{S_n(\phi-\sigma r)}u)(h_{\omega}y)\, d(x,y)^{\alpha}$.
For last, the term \eqref{invariance6} is bounded by 
\begin{eqnarray*}
2^{1-\alpha} C_7^\alpha \,|b|^{\alpha} \, (e^{S_n(\phi-\sigma r)}|v|)(h_{\omega}x) \, d(x,y)^{\alpha}\leqslant 2^{1-\alpha} C_7^\alpha \,e^{{n|\phi-\sigma r|}_{C^{\alpha}(I)}}\,|b|^{\alpha}(e^{S_n(\phi-\sigma r)}u)(h_{\omega}y).
\end{eqnarray*}
Now, using that $|f_{\sigma}|_{\infty}\leqslant 2|f_0|_{\infty},\, \,|f^{-1}_{\sigma}|_{\infty}\leqslant 2|f^{-1}_0|_{\infty}$ and $\chi$ takes values on the interval $(\frac{1}{2},1]$ we deduce that
\begin{eqnarray*}
|A_{s,h_{\omega},n}(f_{\sigma}v)(x)-A_{s,h_{\omega},n}(f_{\sigma}v)(y)|\leqslant 8C|f_0|_{\infty}|f^{-1}_0|_{\infty}|b|^{\alpha} 
	\, A_{\sigma,h_{\omega},n}(f_{\sigma}\chi u)(y) \, d(x,y)^{\alpha}.
\end{eqnarray*} 
The same reasoning ensures that \eqref{invarcone2} is bounded above by 
$
32C|f_0|^{2}_{\infty}|f^{-1}_0|^{2}_{\infty}|b|^{\alpha}u_n(y)d(x,y)^{\alpha}.
$
and \eqref{invarcone2.1} is bounded above by
$8|f_0|_{C^{\alpha}(I)}|f^{-1}_0|_{\infty}|b|^{\alpha}u_n(y)d(x,y)^{\alpha}.$
Altogether we obtain $|v_n(x)-v_n(y)|\leqslant C_0 |b|^{\alpha}u_n(y)d(x,y)^{\alpha}$, which ensures that $(u_n,v_n)\in \mathcal C_b$. This completes the proof
of the lemma.
\end{proof}
The following lemma furnishes the last ingredient for the proof of Proposition~\ref{mainprop}.

\begin{lemma}\label{L2contrac}
There are $\vep>0$ and $0<\tau<1$   such that
\begin{eqnarray*}
\int |\mathscr  L^{mn}_sv|^2d\mu_\sigma 
	\leqslant  \tau^m\|v\|_{L^{\infty}(\mu_\sigma)}^2
\end{eqnarray*}
for all $m\geqslant 1,\,s=\sigma+ib,\, |\sigma|<\varepsilon,\, |b|\geqslant b_1$, and all $v\in C^{\alpha}(I)$ satisfying $|v|_{C^{\alpha}(I)}\leqslant C_0|b|^{\alpha}\|v\|_{L^{\infty}(\mu_\sigma)}$.
\end{lemma}

\begin{proof} 
By Lemma \ref{invarcone}, taking $u_0\equiv 1,\,\, v_0=v/|v|_{\infty}$, we find a sequence $(u_m,v_m)\in \mathcal C_b$ where
$u_{m+1}=\mathscr L^{n}_{\sigma}(\chi_mu_m),$ 
$v_{m+1}=\mathscr L^{n}_s(v_m),$
and $\chi_m=\chi(b,u_m,v_m)$ is a bump-function defined as in \eqref{def-xi}. 
Thus, in view of Corollary~\ref{corollarycancel}, it is enough to prove that there exists $ \tau \in (0,1)$, independent of $b$, such that 
\begin{eqnarray}
\int {u^2_{m+1}}d\mu_\sigma
	\leqslant \tau \int{u^2_{m}}d\mu_\sigma \,\,\,\,\,\,\textit{for all}\,\,\,\,\, m\geqslant0. 
\end{eqnarray}
By definition
\begin{eqnarray*}
u_{m+1}&=&{\lambda^{-n}_\sigma}{f_{\sigma}^{-1}}\displaystyle\sum_{\omega\in\mathcal{P}^{(n)}}e^{S_{n}(\phi-\sigma r)\circ h_{\omega}}(f_{\sigma}{\chi}_mu_m)\circ h_{{\omega}}\\
&=&{\lambda^{-n}_\sigma} {f_{\sigma}^{-1}}\displaystyle\sum_{\omega\in\mathcal{P}^{(n)}} \, \big[(e^{S_{n}\phi}f_{\sigma})^{1/2}u_m
\big]\circ h_{\omega} \,\cdot\, [e^{(\frac{S_{n}\phi}{2}-\sigma S_{n} r)}(f^{1/2}_{\sigma}{\chi}_m)]\circ h_{\omega}.
\end{eqnarray*}
By Cauchy-Schwartz inequality
\begin{eqnarray}
u_{m+1}^2&\leqslant 
	&(\lambda_{\sigma}^{n}{f_{\sigma}})^{-2}\displaystyle\sum_{\omega\in\mathcal{P}^{(n)}}(e^{S_{n}\phi}f_{\sigma}{u_m^2})\circ h_{\omega} \, \cdot \, \displaystyle\sum_{\omega\in\mathcal{P}^{(n)}}e^{S_{n}(\phi-2\sigma r)\circ h_{\omega}}(f_{\sigma}{\chi}^2_m)\circ h_{\omega} \nonumber \\
&\leqslant &\xi(\sigma) \mathscr L^{n}_0(u^2_m)\mathscr L^{n}_{2\sigma}(\chi^2_m),
\label{eqCS}
\end{eqnarray}
where $\xi(\sigma):=(\lambda_{2\sigma}\lambda^{-2}_{\sigma})^{n}\left|\frac{f_\sigma}{f_0}\right|_{\infty}
\left|\frac{f_\sigma}{f_{2\sigma}}\right|_{\infty}\left|\frac{f_0}{f_\sigma}\right|_{\infty}\left|\frac{f_{2\sigma}}{f_\sigma}\right|_{\infty}$
tends to 1 as $\sigma$ tends to zero. 

\smallskip
Now, recall that according to \eqref{def-xi}, the bump functions $\chi_m : I \to [\eta,1]$ are constructed in a way that
one can decompose the interval as $I=\hat{I}_m\cup\hat{J}_m$ where $\hat I_m$ is a collection of sub-intervals determined by cancellations:  
$
\hat{I}_m=\bigcup_{j=1}^{\kappa_m} \hat{I}_{\tilde\omega_j}
$ 
where each interval $\hat{I}_{\tilde\omega_j}$ is the middle-third sub-interval of $B(x_i, \delta/2|b|)$, 
 the collection $(\tilde\omega_j)_{1\leqslant j \leqslant \kappa_m}$ of winning inverse branches in 
$\mathcal P^{(n)}$ is determined by
~\eqref{omegawinners} with $u=u_m$ and $v=v_m$, and
\begin{enumerate}
\item[(i)] $\chi_m (h_{\tilde \omega_j}(x))=\eta$, for every $x\in \hat{I}_{\tilde\omega_j}$ and $j=1,2, \dots \kappa_m$,
	\smallskip
\item[(ii)] $\chi_m (x)=1$, for every $x\notin \bigcup_{j=1}^{\kappa_m} h_{\tilde \omega_j}(B(x_j, \delta/2|b|))$, 
	\smallskip
\item[(iii)] $|\chi_m'|\leqslant |b|$
\end{enumerate}
(recall the description in Figure~2 (B)), and 
$
\hat J_m 
	=\bigcup_{j=1}^{\kappa_m} \, Q_j\setminus \hat{I}_{\tilde\omega_j}.
$
Using that $\mathscr L_{2\sigma} 1=1$ we conclude that if $y\in \hat{I}_{\tilde\omega_j}$ then
\begin{eqnarray*}
(\mathscr L^{n}_{2\sigma}\chi^2_m)(y)&\leqslant & \frac{1}{\lambda^{n}_{2\sigma}{f_{2\sigma}(y)}}
	\; \Big[\eta^2e^{(S_{n}(\phi-2\sigma r)\circ h_{\tilde\omega_j}(y)}f_{2\sigma}(h_{\tilde\omega_j}(y))\\
&&\qquad \qquad\qquad+\sum_{\substack{\omega\in\mathcal{P}^{(n)} \\ \omega\neq {\tilde\omega_j}}}e^{(S_{n}(\phi-2\sigma r) \circ h_{\omega)}(y)}f_{2\sigma}(h_{\omega}(y)) \Big]\\
&\leqslant &1-(1-\eta^2)2^{-(n+2)}|f_0|^{-1}_{\infty}\inf f_0 \,e^{-\sup_{\omega\in \mathcal P^{(n)}} |S_{n}(\phi-2\sigma r)\circ h_{\omega}|_{\infty}}=:\eta_1<1.
\end{eqnarray*}
Since $|\mathscr L^{n}_{2\sigma}(\chi^2_m)(y)|\leqslant 1$ for every $y\in I$, we deduce from ~\eqref{eqCS} that
\begin{eqnarray}\label{intcancel}
\int u_{m+1}^2 d\mu_\sigma
	\leqslant \xi(\sigma)\left(\eta_1\int_{\hat{I}_m}\mathscr L^{n}_0(u^2_m) d\mu_\sigma
+\int_{\hat{J}_m} \mathscr L^{n}_0(u^2_m)d\mu_\sigma\right).
\end{eqnarray}
At this point we use that the Gibbs measure gives to the region $\hat{I}_m$ 
a definite proportion of the interval and that such proportion (independent of the quantities $b$, $n$ or $m$). More precisely, 
Lemma ~\ref{le:cancellations-measure} and Corollary~\ref{comparijpropfinite} (applied with $\delta/2|b|$ instead of $\delta/2|b|$ and 
taking the function $\mathscr L^{n}_0(u_m^2)$) guarantee that there exists $\delta'>0$ (independent of $b$, $n$ and $m$) satisfying
\begin{eqnarray}\label{comparij}
\int_{\hat{I}_m} \mathscr L^{n}_0(u_m^2) d\mu_\sigma\geqslant {\delta}'\int_{\hat{J}_m} \mathscr L^{n}_0(u_m^2)\, d\mu_\sigma.
\end{eqnarray} 
This ensures that
\begin{eqnarray*}
\eta_1\int_{\hat{I}_m} \mathscr L^{n}_0(u_m^2)\, d\mu_\sigma+\int_{\hat{J}_m} \mathscr L^{n}_0(u_m^2)\, d\mu_\sigma
	\leqslant \tau' \int_I \mathscr L^{n}_0(u_m^2)\, d\mu_\sigma
\end{eqnarray*}
\big(taking $\tau'=\dfrac{1+\eta_1\delta'}{1+\delta'}<1$\big),
which together with \eqref{intcancel}, ultimately leads to
\begin{eqnarray*}
\int u_{m+1}^2 d\mu_\sigma &\leqslant &\xi(\sigma)\left(\eta_1\int_{\hat{I}_m}\mathscr L^{n}_0(u^2_m) d\mu_\sigma+\int_{\hat{J}_m} \mathscr L^{n}_0(u^2_m)d\mu_\sigma\right)\\
&\leqslant & \xi(\sigma)\tau' \int_I \mathscr L^{n}_0(u^2_m)d\mu_\sigma=\xi(\sigma)\tau' \int_I u^2_m d\mu_\sigma.
\end{eqnarray*}
Hence it is enough to take $\varepsilon\in(0,1)$  small so that $ \tau :=\xi(\sigma)\tau'<1$ for every $|\sigma|<\vep$.
\end{proof}

We are now in position to complete the proof of the proposition. 

\begin{proof}[Proof the Proposition \ref{mainprop}] 
Choose $b_3>b_2$ so that $k:=\lfloor \beta \log b\rfloor \gg n$ for every $|b|>b_3$ and
recall that $\beta>0$ was chosen so that
$
\lambda^{\frac{\alpha k}{16}} \leqslant C_0 |b|^\alpha
$
for every $|b|>b_2$.
Hence, any  $v\in C^{\alpha}(I)$ so that $\|v\|_{(b)}< \lambda^{\frac{\alpha k}{16}} \|v\|_{L^{\infty}(\mu_{\sigma})}$ also satisfies $\|v\|_{(b)}< C_0 |b|^\alpha \|v\|_{L^{\infty}(\mu_{\sigma})}$. 
Now, combining H\"older's inequality with Lemma~\ref{L2contrac}, we deduce that 
\begin{eqnarray*}
\|{\mathscr L}_s^kv\|_{L^{1}(\mu_\sigma)}
	\leqslant \|{\mathscr L}_s^kv\|_{L^{2}(\mu_\sigma)}
	\leqslant  \tau^{\lfloor \frac{k}{n}\rfloor} \, \|v\|_{L^{\infty}(\mu_\sigma)}^2.
\end{eqnarray*}
In particular, there exists $\xi>0$ (independent of $b$) so that $\|{\mathscr L}_s^kv\|_{L^{1}(\mu_\sigma)} \leqslant  \lambda^{-\xi k} \, \|v\|_{L^{\infty}(\mu_\sigma)}^2$, as claimed.
This completes the proof of the proposition.
\end{proof}

\subsection*{Acknowledgments}
The authors are grateful to O. Butterley for helpful comments.
This work is part of the first author's PhD thesis carried out at Federal University of Bahia (Salvador - Bahia). 
DD was supported by CAPES-Brazil. PV was partially supported by CMUP (UID/MAT/00144/2019), which is funded by FCT with national (MCTES) and European structural funds through the programs FEDER, under the partnership agreement PT2020, and 
by Funda\c c\~ao para a Ci\^encia e Tecnologia (FCT) - Portugal through the grant CEECIND/03721/2017 of the Stimulus of
Scientific Employment, Individual Support 2017 Call.

%
%


\begin{thebibliography}{98}
\providecommand{\natexlab}[1]{#1}
\providecommand{\url}[1]{\texttt{#1}}
\expandafter\ifx\csname urlstyle\endcsname\relax
  \providecommand{\doi}[1]{doi: #1}\else
  \providecommand{\doi}{doi: \begingroup \urlstyle{rm}\Url}\fi
  


\bibitem{ABV}
{V. Ara\'ujo, O. Butterley and P. Varandas, Open sets of Axiom A flows with exponentially mixing attractors. 
\textit{Proc. Amer. Math. Soc.} {144} (2016) 2971-2984.}

\bibitem{AM}
V. Ara\'ujo and I. Melbourne, Exponential decay of correlations for nonuniformly hyperbolic flows with a $C^{1+\alpha}$ stable foliation, 
including the classical Lorenz attractor, \emph{Ann. Henri Poincar\'e}, 17 (2016) 2975--3004.

\bibitem{AMV}
V. Ara\'ujo, I. Melbourne and P. Varandas, 
Rapid mixing for the Lorenz attractor and statistical limit laws for their time-1 map, 
\emph{Comm. Math. Phys.} 340 (2015) 901--938.

\bibitem{ArVar}
V.~Ara{\'u}jo and P.~Varandas.
\newblock Robust exponential decay of correlations for singular-flows.
\newblock {\em  Comm. Math. Phys.}, 311  (2012) 215--246.

\bibitem{AGY}
A.~\'Avila, S.~Gou{\"e}zel, and J.-C. Yoccoz.
\newblock {Exponential mixing for the Teichm{\"u}ller flow}.
\newblock {\em {Publ. Math. Inst. Hautes {\'E}tudes Sci.}}, {104} ({2006}) {143--211}.


\bibitem{BV}
V.~Baladi and B.~Vall{\'e}e.
\newblock {Exponential decay of correlations for surface semi-flows without
  finite Markov partitions}.
\newblock {\em {Proc. Amer. Math. Soc.}}, {133}:{3} ({2005}) {865--874}.


\bibitem{Bo73}
R.~Bowen.
\newblock {Symbolic dynamics for hyperbolic flows}.
\newblock {\em {Amer. J. Math.}}, {95} ({1973}) {429--460}.

\bibitem{Bo75}
R.~Bowen.
\newblock {\em {Equilibrium states and the ergodic theory of Anosov
  diffeomorphisms}}, volume {470} of {\em {Lect. Notes in Math.}}
\newblock {Springer Verlag}, {1975}.

\bibitem{BR75}
R.~Bowen and D.~Ruelle.
\newblock The ergodic theory of {A}xiom {A} flows.
\newblock {\em Invent. Math.}, 29 (1975) 181--202.

\bibitem{Bo78}
R.~Bowen.
\newblock {Markov partitions are not smooth}.
\newblock {\em {Proc. Amer. Math. Soc.}}, 71:1 (1978) 130--132.


%


\bibitem{BW} {O. Butterley and K. War}. 
Open sets of exponentially mixing Anosov flows,
\emph{J. Eur. Math. Soc.}, 22 (2020) 2253--2285. 
 
\bibitem{DV} D. Daltro and P. Varandas, 
Exponential decay of correlations for Gibbs measures and semiflows over $C^{1+\alpha}$ piecewise expanding maps, 
\emph{Ann. Henri Poincar\'e} (2021). https://doi.org/10.1007/s00023-020-00991-5


\bibitem{DP}
M. Denker and W. Philipp. Approximation by Brownian motion for Gibbs measures and flows under a function. 
\emph{Ergod. Th. \& Dynam. Sys.} 4 (1984) 541--552.



\bibitem{chernov98}
N.~I. Chernov.
\newblock {Markov approximations and decay of correlations for Anosov flows}.
\newblock {\em {Ann. of Math. (2)}}, {147}({2}):{269--324}, {1998}.

\bibitem{D}
D.~Dolgopyat.
\newblock {On decay of correlations in Anosov flows}.
\newblock {\em {Ann. of Math. (2)}}, {147}:{2} ({1998}) {357--390}.

\bibitem{dolgopyat98}
D.~Dolgopyat.
\newblock {Prevalence of rapid mixing in hyperbolic flows}.
\newblock {\em {Ergod. Th. \& Dynam. Sys.}}, {18}:{5} ({1998}) {1097--1114}.





\bibitem{FMT}
M.~Field, I.~Melbourne, and A.~T{\"o}rok.
\newblock {Stability of mixing and rapid mixing for hyperbolic flows}.
\newblock {\em {Ann. of Math. (2)}}, {166} ({2007}) {269{--}291}.

%
%






\bibitem{Go}
S. Gou\"ezel.
Local limit theorem for nonuniformly partially hyperbolic skew-products and Farey sequences
 \emph{Duke Math. J.} 147:2 (2009) 193--284.

\bibitem{GS}
S. Gou\"ezel and L. Stoyanov.
 Quantitative Pesin theory for Anosov diffeomorphisms and flows.
 \emph{Ergod. Th. Dynam. Sys.} 39 (2019) 159--200.

\bibitem{li}
C.~Liverani.
\newblock {On contact Anosov flows}.
\newblock {\em {Ann. of Math. (2)}}, {159}:{3} ({2004}){1275--1312}.

\bibitem{MT}
I. Melbourne and A. T\"or\"ok, 
Central limit theorems and invariance principles for time-one maps of hyperbolic
flows. \emph{Commun. Math. Phys.} 229:1 (2002) 57--71.

%

\bibitem{Plante}
J. Plante.
Anosov flows. \emph{Amer. J. Math.}, 94:3 (1972) 729--754.

\bibitem{Pol84}
M.~Pollicott.
\newblock {A complex Ruelle-Perron-Frobenius theorem and two counterexamples}.
\newblock {\em {Ergod. Th. \& Dynam. Sys.}} {(1984)} {135--146}.

\bibitem{Pol85}
M.~Pollicott.
\newblock {On the rate of mixing of Axiom A flows}.
\newblock {\em {Invent. Math.}}, {81}:3 {(1985)} {413--426}.

\bibitem{P}
M. Pollicott,  {On the mixing of Axiom A attracting flows and a conjecture of Ruelle}. 
{\em {Ergod. Th. \& Dynam. Sys.}} {19} (1999) 535--548. 

\bibitem{PS}
M. Pollicott and R. Sharp,  
Exponential error terms for growth functions of negatively curved surfaces,
\emph{Am. J. Math.} 120 (1998) 1019--1042.






\bibitem{Rat}
M. Ratner. 
\newblock Markov partitions for Anosov flows on $n$-dimensional manifolds.
\emph{Isr. J. Math.} 15 (1973) 92--114.

\bibitem{Ru76b}
D.~Ruelle.
\newblock A measure associated with {A}xiom {A} attractors.
\newblock {\em Amer. J. Math.}, 98 (1976) 619--654.


\bibitem{ruelle1983}
D.~Ruelle.
\newblock {Flots qui ne m{\'e}langent pas exponentiellement}.
\newblock {\em {C. R. Acad. Sci. Paris S{\'e}r. I Math.}},
  {296}:{4} ({1983}) {191--193}.

\bibitem{Ru86}
D. Ruelle. 
Locating resonances for AxiomA dynamical systems. 
\emph{J. Stat. Phys.}, 44  (1986) 281--292.


\bibitem{Si72}
Ya. Sinai.
\newblock Gibbs measures in ergodic theory.
\newblock {\em Russian Math. Surveys}, 27 (1972) 21--69.



\bibitem{Sto}
L. Stoyanov,
Spectra of Ruelle transfer operators for Axiom A flows,
\emph{Nonlinearity} 24 (2011) 1089--1120



\bibitem{St1}
L. Stoyanov,
Pinching conditions, linearization and regularity of Axiom A flows, 
\newblock {\em Discrete Cont. Dynam. Sys.}, 33:2 (2013) 391--412.


\bibitem{St0}
L. Stoyanov,
Spectral properties of Ruelle transfer operators for regular Gibbs measures and decay of correlations for contact Anosov flows, 
\emph{Memoirs Amer. Math. Soc.} (to appear)

\bibitem{T0}
M. Tsujii, Decay of correlations in suspension semi-flows of angle-multiplying maps. \emph{Ergod. Th. Dynam. Sys.}, 28:1 (2008) 291--317. 

\bibitem{T}
{M. Tsujii, Exponential mixing for generic volume-preserving Anosov flows in dimension three, 
\emph{J. Math. Soc. Japan.} {70}  (2018) 757-821.}%

\bibitem{TZ}
M. Tsujii and Z. Zhang, 
Smooth mixing Anosov flows in dimension three are exponential mixing, 
Preprint arXiv:2006.04293 

%


\end{thebibliography}
\end{document}